\documentclass[preprint,11pt]{elsarticle}

\usepackage[utf8]{inputenc}      % source file encoding
\usepackage[T1]{fontenc}         % correct hyphenation & accents
\usepackage[english]{babel}

\usepackage{microtype}           % subtle kerning & protrusion

\usepackage{amsmath,amssymb,amsfonts}
\usepackage{stmaryrd}            % extra brackets (⟦ ⟧)
\usepackage{yfonts}              % ℳ, �� …
\usepackage{mathrsfs}
\usepackage{esint}
\usepackage{algorithm}      % floating environment
\usepackage{algpseudocode}
\usepackage{booktabs}
\usepackage{adjustbox}  
\usepackage{soul}
\usepackage{tikz-cd}
\usepackage{hyperref}
\usepackage{amsthm}

\theoremstyle{definition}

\theoremstyle{remark}

\usepackage{graphicx,epstopdf}
\usepackage{pgfplots}
\pgfplotsset{compat=newest}

\usepackage{booktabs,multirow,array,diagbox}
\usepackage{tabularx}   % <- needed for the X & Y column types
\usepackage{colortbl}
\usepackage{caption}
\usepackage[margin=3cm]{geometry}
\usepackage{pdflscape}

\usepackage{enumitem}
\setlist[itemize]{label=–,leftmargin=2em}

\usepackage[dvipsnames]{xcolor}
\usepackage[colorinlistoftodos]{todonotes}
\usepackage[most]{tcolorbox}

\usepackage{url} 
\usepackage[nameinlink,capitalise]{cleveref}
\crefname{equation}{equation}{equations}
\crefname{theorem}{theorem}{theorems}
\crefname{lemma}{lemma}{lemmas}
\crefname{proposition}{proposition}{propositions}
\crefname{corollary}{corollary}{corollaries}
\crefname{definition}{definition}{definitions}
\crefname{example}{example}{examples}
\crefname{remark}{remark}{remarks}

\usepackage[displaymath,mathlines]{lineno}
\usepackage{standalone}
\usepackage{subfiles} 
\usepackage{comment}

\RequirePackage{bm}        % better \boldsymbol for Greek etc.
\RequirePackage{tabularx}  % gives the X column used below

\newcommand{\vertiii}[1]{%
  \left\vert\kern-0.25ex\left\vert\kern-0.25ex\left\vert
  #1
  \right\vert\kern-0.25ex\right\vert\kern-0.25ex\right\vert}

\newcolumntype{Y}{>{\centering\arraybackslash}X}

\newtheorem{prop}{Proposition}[]
\newtheorem{asm}{Assumption}[]
\journal{CMAME}

\begin{document}

\begin{frontmatter}

%% Title, authors and addresses

%% use the tnoteref command within \title for footnotes;
%% use the tnotetext command for theassociated footnote;
%% use the fnref command within \author or \affiliation for footnotes;
%% use the fntext command for theassociated footnote;
%% use the corref command within \author for corresponding author footnotes;
%% use the cortext command for theassociated footnote;
%% use the ead command for the email address,
%% and the form \ead[url] for the home page:
%% \title{Title\tnoteref{label1}}
%% \tnotetext[label1]{}
%% \author{Name\corref{cor1}\fnref{label2}}
%% \ead{email address}
%% \ead[url]{home page}
%% \fntext[label2]{}
%% \cortext[cor1]{}
%% \affiliation{organization={},
%%             addressline={},
%%             city={},
%%             postcode={},
%%             state={},
%%             country={}}
%% \fntext[label3]{}

\title{Multiscale Mixed-Dimensional Simulation via Domain Decomposition and Non-Intrusive Neural Model Order Reduction}

%% use optional labels to link authors explicitly to addresses:
%% \author[label1,label2]{}
%% \affiliation[label1]{organization={},
%%             addressline={},
%%             city={},
%%             postcode={},
%%             state={},
%%             country={}}
%%
%% \affiliation[label2]{organization={},
%%             addressline={},
%%             city={},
%%             postcode={},
%%             state={},
%%             country={}}

\author[label1]{Nunzio Dimola}
\author[label1]{Piermario Vitullo}
\author[label1]{Paolo Zunino} %% Author name

%% Author affiliation
\affiliation[label1]{organization={MOX, Department of Mathematics, Politecnico di Milano},%Department and Organization
            %addressline={}, 
            city={Milan},
            %postcode={}, 
            %state={},
            country={Italy}
            }

%% Abstract
\begin{abstract}
%% Text of abstract
Many computational models arising in science and engineering exhibit a multiscale structure that makes the assembly or direct solution of the global problem computationally prohibitive. Domain Decomposition (DD) methods overcome this limitation by replacing the global problem with a sequence of coupled local problems, whose iterative solution reconstructs the global response. This work introduces a method in the family of Domain Decomposition Reduced Order Models (DD-ROMs), based on the observation that DD naturally localizes not only the solution operator but also its geometric and parametric dependence. The central idea is that DD transforms a globally intractable solution map into a family of locally representable operators learnable from affordable local data after identification with a common reference configuration, a concept that we formalize through the notion of \textit{local representability}.

Non-intrusive neural surrogates are then trained to approximate the fine-scale local operations and embedded into the iterative solver. The training algorithm is based on a cascaded strategy designed to match the distributions encountered by the deployed surrogate iteration. We interpret the resulting DD method as a perturbed fixed-point iteration and establish that the global error remains bounded by the surrogate approximation error. The framework is instantiated for mixed-dimensional elliptic problems coupling three-dimensional bulk domains with embedded one-dimensional inclusions, using a two-level non-overlapping Robin–Robin method. Numerical experiments show that the resulting DD-ROM is stable, achieves accurate approximation on unseen microscale geometries and features good scalability properties with respect to the number of subdomains, scaling to large size global problems while avoiding fine-scale operator assembly and local high-fidelity solvers in the online stage.
\end{abstract}

%%Graphical abstract
%\begin{graphicalabstract}
%\includegraphics{grabs}
%\end{graphicalabstract}

%%Research highlights
%\begin{highlights}
%\item Research highlight 1
%\item Research highlight 2
%\end{highlights}

%% Keywords
%\begin{keyword}
%% keywords here, in the form: keyword \sep keyword

%% PACS codes here, in the form: \PACS code \sep code

%% MSC codes here, in the form: \MSC code \sep code
%% or \MSC[2008] code \sep code (2000 is the default)

%\end{keyword}

\end{frontmatter}

%% Add \usepackage{lineno} before \begin{document} and uncomment 
%% following line to enable line numbers
%% \linenumbers

%% main text
%%

\section{Introduction}
\label{sec:introduction}

Many problems in science and engineering involve localized fine-scale structures interacting through a much larger surrounding domain. Examples include fractured porous media \cite{Peaceman1978183,jaffre2012modeling}, plant root systems \cite{koch2018new,Schroder2012}, vascularized biological tissues \cite{Possenti2019,Possenti20213356,vidotto2019hybrid,Kremheller2019}, and composite materials \cite{Khristenko2021,Firmbach2023,EDELVIK2024116593}. Their numerical simulation leads to extremely large multiscale systems whose repeated solution is computationally prohibitive. In many cases, neither globally resolved training simulations nor the assembly of the global fine-scale operator are computationally affordable.

Several complementary methodologies have been developed to address this challenge. Multiscale discretization techniques, including homogenization \cite{bensoussan2011asymptotic}, multiscale finite element methods \cite{efendiev2009multiscale}, heterogeneous multiscale methods \cite{cms/1118150402}, localized orthogonal decomposition \cite{Henning2014A1609}, and numerical upscaling strategies \cite{EDELVIK2024116593}, derive effective coarse-scale models while retaining the influence of unresolved fine structures. Projection-based model order reduction \cite{benner2015survey,quarteroni2015reduced} instead exploits the low-dimensional structure of solution manifolds to accelerate repeated simulations. More recently, data-driven and operator-learning approaches have demonstrated the ability to approximate solution operators directly from data \cite{doi:10.1142/S0218202525500125,Berner,kovachki2024operator,boulle2024mathematical}. Despite their differences, these approaches generally rely on either problem-specific multiscale constructions or on globally resolved high-fidelity datasets generated during an expensive offline stage. For very large multiscale systems, this offline cost may itself become the dominant computational bottleneck.

Domain decomposition provides a natural mechanism to shift the reduction task from the global system to a family of local problems \cite{quarteroni2017domain,Toselli2005}. This idea is well established in localized and component-based model reduction. Classical approaches construct reduced spaces for subdomains or reusable archetype components and couple them through Lagrange multipliers, static condensation, reduced interfaces, discontinuous formulations, optimization principles, or Schwarz-type iterations. A recent survey \cite{RuanClassRozza2026} organizes these methods, generally called DD-ROMs into intrusive projection-based and non-intrusive data-driven classes and emphasizes the central roles of spatial localization, local parameterization, local basis construction, and interface coupling. Examples of these ideas are for instance provided in \cite{DiscacciatiHesthaven2023,TADDEI2024113038,FarcasGundeviaMunipalliWillcox2024}, showing that the combination of domain decomposition and local model reduction is an active research direction.
More generally, the rise of scientific machine learning is fostering an increasingly close interaction between domain-decomposition techniques and data-driven modeling, leading to hybrid reduced-order frameworks that combine algorithmic structure with learned local representations \cite{heinlein2021combining}.

The contribution of the present work lies in an operator-level realization of DD-ROM designed for multiscale mixed-dimensional problems. We decompose the fine-scale operations entering a two-level non-overlapping Robin--Robin iteration and approximate them by introducing three local surrogates. These are the transmission map \(T_{\rm rom}\), the local solution map \(S_{\rm rom}\), and the local contribution \(C_{\rm rom}\) to the coarse residual.
%The local solution surrogate \(S_{\rm rom}\) combines a proper-orthogonal-decomposition space with a mesh-informed nonlinear closure driven by the local microscale descriptor. The resulting approximation can be interpreted as a multiscale enrichment of a common POD space. 
%Their combined use replaces all fine-scale operator evaluations required by the online domain decomposition iteration. 
%We refer to this particular two-level, operator-surrogate construction as the multiscale DD-ROM framework (MS-DD-ROM).
%MS-DD-ROM exhibits both a non-intrusive nature and scalability property. 
%A surrogate for the local subproblem alone would still require high-fidelity evaluations of interface transmission quantities and coarse-residual contributions during the online phase. By learning these additional local algorithmic maps, the proposed method removes the remaining dependence on fine-scale local operators. At the same time, retaining an explicit coarse correction preserves a global communication mechanism whose role becomes increasingly important as the number of subdomains grows. 
Their combined use replaces all fine-scale operator evaluations required by the online domain-decomposition iteration. A surrogate for the local subproblem alone would still require high-fidelity evaluations of interface transmission quantities and coarse-residual contributions during the online phase. By learning these additional local algorithmic maps, the proposed method removes the remaining dependence on fine-scale local operators. At the same time, retaining an explicit coarse correction preserves a global communication mechanism, whose role becomes increasingly important as the number of subdomains grows. The resulting approach is both non-intrusive and scalable. 

We refer to this two-level operator-surrogate construction as the multiscale DD-ROM framework, or MS-DD-ROM.

Localization also changes the nature of the parameter space. Each local problem depends on physical and geometric information, including the embedded one-dimensional microstructure, as well as on interface data generated by the iterative solver. The localized parameter is therefore partly physical and partly algorithmic, and its distribution evolves when exact local operators are progressively replaced by learned ones. This motivates the \textit{local representability} viewpoint developed in this paper: after mapping all subdomains to a common reference configuration, the corresponding local solution maps are assumed to belong to a common parametrized family. It also motivates a cascaded training strategy in which the surrogates are introduced sequentially and each successive model is trained on states generated by the already deployed approximations.

Replacing the exact local operators by neural surrogates perturbs the fixed-point iteration underlying the domain decomposition method. We address this issue by interpreting the MS-DD-ROM as an inexact contraction and by relating the asymptotic global error to the contraction factor of the exact method and to the local surrogate errors. This analysis provides the conceptual link between local approximation quality, the training distribution, and the stagnation level of the global iteration.

The framework is instantiated for a class of mixed-dimensional elliptic problems coupling a three-dimensional continuum with an embedded one-dimensional network. The three-dimensional problem is solved by a two-level non-overlapping Robin--Robin method, while the one-dimensional structure enters the local parameterization and drives the microstructure-informed surrogate models. Although this setting is motivated by oxygen transport and perfusion in vascularized tissue \cite{Possenti20213356}, the methodology is formulated at an abstract level and is applicable more broadly to parametrized multiscale PDEs admitting reusable local decompositions.

\section{A localized learning framework for multiscale parametrized PDEs}\label{sec:MultiScaleAbstract}
%\section{\textit{Global} multiscale simulation via \textit{local} ROMs}

\medskip\noindent
We consider a particular family of problems governed by parametrized PDEs whose coefficients and source terms exhibit a \emph{small-scale structure (or microscale)}, namely fine heterogeneities that induce the coexistence of multiple characteristic length scales in the solution. 
Let $\Omega \subset \mathbb{R}^d$ denote the \emph{global domain (or macroscale)}, with a characteristic size $L = \operatorname{diam}(\Omega)$, encompassing a structure $\Lambda$ characterized by a \emph{local feature}, that is, a geometric variation occurring on a smaller scale $\ell \ll L$.

Our aim is to develop an efficient solution strategy capable of accurately capturing both the global and local features of the solution in a parametrized setting. We define a parametrized problem $\mathfrak{p}^\mu$ posed in the two-scale configuration $(\Omega, \Lambda)$, where the parameter $\mu \in \mathcal{P}$ encodes the variability induced by the underlying structures $\Lambda$, represented by a parameterization associated with the parameter space $\mathcal{P}$. We focus on steady elliptic problems of the form: 
\begin{equation}\label{eq:split-param}
\textrm{find $u^\mu \in \mathcal{V}(\Omega)$ such that} \
\mathfrak{p}^\mu(u^\mu) = 0.
\end{equation}
At fixed \(\mu\in\mathcal P\), the corresponding solution map is defined as
\(\mathfrak s:\mathcal P\to \mathcal V(\Omega),
\ \mathfrak s(\mu)=u^\mu\).

The global discretization of~\eqref{eq:split-param} generally requires resolving the fine-scale structures throughout $\Omega$, making both the assembly of the discrete operator and the repeated solution of the resulting high-fidelity system prohibitively expensive. The purpose of this section is therefore to formulate an abstract framework in which the global solution operator is replaced by a collection of localized operators that can be learned from affordable local simulations.

\subsection{Localized operator decomposition}
\label{sec:localized_operator_decomposition}

Domain decomposition replaces the direct evaluation of the global solution map by an iterative composition of local parametrized operators coupled through interface conditions. Let
\(
\Omega=\cup_{i\in\mathcal I}\Omega_i,
\
\mathcal I:=\{1,\ldots,N_{\rm sub}\},
\)
be a partition of the computational domain into non-overlapping subdomains. At an abstract level, the corresponding domain decomposition iteration is written as
\(
u^{(k+1)}
=
\mathcal{DD}^{\mu}\!\left(u^{(k)}\right),
\
k=0,1,\ldots,
\)
where
\(
\mathcal{DD}^{\mu}:\mathcal V(\Omega)\rightarrow\mathcal V(\Omega)
\)
denotes the iteration operator associated with the global problem~$\mathfrak p^\mu$. Under standard assumptions on the transmission conditions, the fixed point of $\mathcal{DD}^{\mu}$ coincides with the solution
\(
u^\mu=\mathfrak s(\mu).
\)

The action of $\mathcal{DD}^{\mu}$ is completely determined by a collection of local solution operators. Denoting by
\(
u_i^{(k)}:=u^{(k)}|_{\Omega_i}
\)
the local state on $\Omega_i$, the update is obtained by solving a localized problem depending on a local parameter
\(
\widehat\mu_i^{(k)},
\)
which contains both the restriction of the physical data and the interface information generated by the current DD iteration. The corresponding exact local solution operator is written as
\(
u_i^{(k+1)}
=
\mathfrak s_i\!(\widehat\mu_i^{(k)}),
\ 
i\in\mathcal I,
\)
and we denote the associated localized problem by
\(
\mathfrak p_i^{\widehat\mu}.
\)

The central idea of this work is to replace these exact local operators by surrogate models,
\(
\widetilde u_i^{(k+1)}
=
\widetilde{\mathfrak s}_i
\!(\widehat\mu_i^{(k)}),
\)
while retaining the global DD iteration as the computational framework,
\(
\widetilde u^{(k+1)}
=
\widetilde{\mathcal{DD}}^\mu
\!(\widetilde u^{(k)}),
\)
where $\widetilde{\mathcal{DD}}^\mu$ is obtained by replacing the exact local operators
$\mathfrak s_i$
with their surrogate counterparts
$\widetilde{\mathfrak s}_i$.
Domain decomposition therefore plays a dual role: it provides the localization required to generate affordable high-fidelity training data, and it supplies the iterative computational framework in which the learned operators are subsequently deployed.

\begin{figure}[h!]
    \centering
    \includegraphics[width=0.9\linewidth]{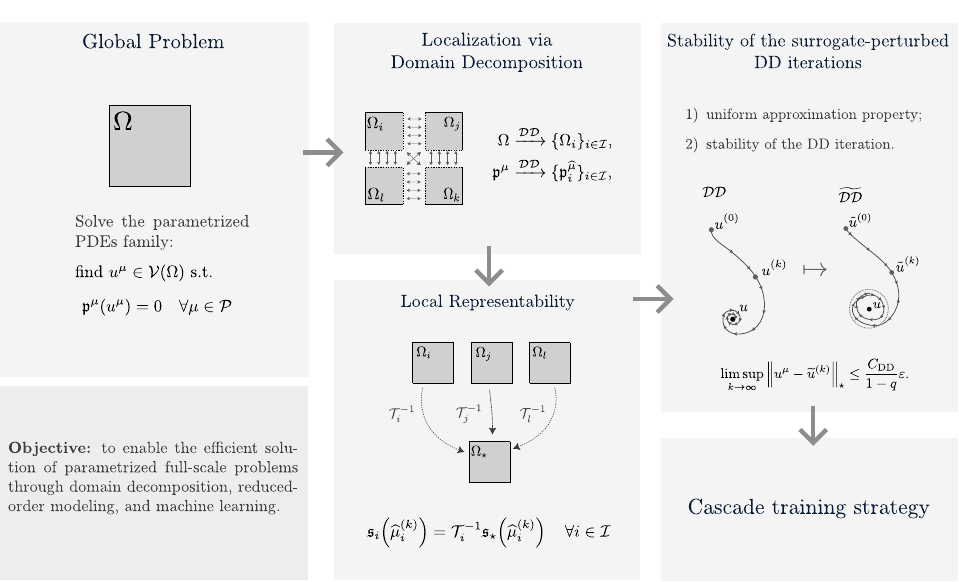}
    \caption{Schematic representation of the proposed framework. Domain decomposition localizes the global multiscale problem into reusable local operators, which are learned on a common reference configuration and subsequently embedded into the global iterative solver.}
    \label{fig:placeholder}
\end{figure}

This construction raises two fundamental questions addressed in the remainder of this section. First, replacing exact local operators by surrogate models perturbs the underlying DD iteration, requiring an analysis of the resulting convergence and stability properties. Second, although localization makes the generation of training data computationally affordable, the corresponding local parameter space must represent not only the variability of the physical configurations but also the interface states generated during the DD iteration. This motivates the introduction of an enlarged localized parameter space $\widehat{\mathcal P}$ and the notion of \emph{local representability}, which constitute the foundation of the proposed MS-DD-ROM framework.

\subsection{Representability on a reference local problem}
\label{sec:local_representability}

The localization of the solution operator naturally induces a localization of its parameter dependence. Unlike the global problem, however, a local subproblem is not completely characterized by the restriction of the global parameter $\mu\in\mathcal P$. It also depends on the interface information generated during the DD iteration. Consequently, each local problem is identified by a \emph{localized parameter}
\(
\widehat\mu_i^{(k)}
=
(\lambda_i,\nu_i^{(k)})
\in
\widehat{\mathcal P},
\)
where $\lambda_i$ describes the local physical configuration (geometry, coefficients, source terms and fine-scale structure), while $\nu_i^{(k)}$ collects the algorithmic quantities exchanged with neighboring subdomains at iteration $k$. The localized parameter space $\widehat{\mathcal P}$ therefore extends the original physical parameter space $\mathcal P$ by incorporating the algorithmic states generated by the DD iteration.

A second consequence of localization is that all local problems become geometrically similar. We therefore introduce a reference domain
\(\Omega_\star=(0,1)^d\),
and assume that each subdomain $\Omega_i$ is obtained from $\Omega_\star$ through a bijective mapping
\(
\Phi_i:\Omega_\star\rightarrow\Omega_i.
\)
 The corresponding pullback
\(
\mathcal T_i:\mathcal V(\Omega_i)\rightarrow\mathcal V(\Omega_\star),
\
\mathcal T_i v=v\circ\Phi_i,
\)
transports every local solution to the reference domain.

This leads to the following notion.

\begin{asm}[Local representability]
\label{asm:local_representability}
A family of parametrized problems
\(
\{\mathfrak p^\mu:\mu\in\mathcal P\}
\)
is said to satisfy the \emph{local representability property} if there exist:
\textit{(i)} a reference domain $\Omega_\star$;
\textit{(ii)} a localized parameter space $\widehat{\mathcal P}$;
\textit{(iii)} a reference family of local problems
\(
\{\mathfrak p_\star^{\widehat\mu}:
\widehat\mu\in\widehat{\mathcal P}\}
\),
with associated solution operator
\(
\mathfrak s_\star:
\widehat{\mathcal P}
\rightarrow
\mathcal V(\Omega_\star),
\)
such that, for every subdomain $i$, every DD iteration $k$, and every global parameter $\mu\in\mathcal P$, there exists a localized parameter
\(
\widehat\mu_i^{(k)}
\in
\widehat{\mathcal P}
\)
satisfying
\[
\mathcal T_i
\mathfrak s_i\!\left(\widehat\mu_i^{(k)}\right)
=
\mathfrak s_\star\!\left(\widehat\mu_i^{(k)}\right)
\ \textrm{or equivalently} \
\mathfrak s_i\!\left(\widehat\mu_i^{(k)}\right)
=
\mathcal T_i^{-1}
\mathfrak s_\star\!\left(\widehat\mu_i^{(k)}\right).
\]
\end{asm}

The local representability property states that every local problem arising during the DD iteration can be interpreted as one realization of a single reference parametrized problem posed on $\Omega_\star$. Consequently, instead of approximating a different solution operator on each subdomain, one can learn a unique surrogate
\[
\widetilde{\mathfrak s}_\star:
\widehat{\mathcal P}
\rightarrow
\mathcal V(\Omega_\star),
\]
which is subsequently reused throughout the DD iteration after pullback to the physical subdomains. The effectiveness of the proposed MS-DD-ROM framework therefore relies on constructing a localized parameter space $\widehat{\mathcal P}$ sufficiently rich to represent both the physical variability of the local configurations and the algorithmic states generated during the DD iterations.

\subsection{Stability of the surrogate-perturbed DD iterations}\label{sec:local_perturbation_theory}
Replacement of exact local solution maps \(\mathfrak{s}_i\) by surrogate approximations \(\widetilde{\mathfrak{s}}_i\) introduces an approximation error directly within the domain decomposition iteration. We therefore interpret the surrogate-based DD algorithm as a perturbed fixed-point iteration and study the propagation of such perturbations at an abstract level. We assume that the operator \(\mathcal{DD}^{\mu}\) is contractive in a norm
\(\|\cdot\|_{\star}\), namely that there exists \(q\in[0,1)\) such that
\[
\left\|
\mathcal{DD}^{\mu}(v)-\mathcal{DD}^{\mu}(w)
\right\|_{\star}
\leq
q\|v-w\|_{\star},
\qquad
\forall v,w\in\mathcal V(\Omega).
\]
The exact solution \(u^\mu=\mathfrak{s}(\mu)\) is assumed to be the unique fixed point \(u^\mu=\mathcal{DD}^{\mu}(u^\mu)\).

Let \(\widetilde{\mathcal{DD}}^{\mu}\) denote the perturbed iteration operator obtained by replacing one or more exact local solution maps \(\mathfrak{s}_i\) with their surrogate counterparts \(\widetilde{\mathfrak{s}}_i\). The corresponding iteration reads
\(\widetilde{u}^{(k+1)} = \widetilde{\mathcal{DD}}^{\mu} (\widetilde{u}^{(k)})\).
Equivalently, we write
\[
\widetilde{u}^{(k+1)}
=
\mathcal{DD}^{\mu}
\left(\widetilde{u}^{(k)}\right)
+
\delta^{(k)},
\
\mathrm{where}
\
\delta^{(k)}
:=
\widetilde{\mathcal{DD}}^{\mu}
\left(\widetilde{u}^{(k)}\right)
-
\mathcal{DD}^{\mu}
\left(\widetilde{u}^{(k)}\right).
\]
Note that \(\delta^{(k)}\) is the perturbation induced at iteration \(k\) by the surrogate local solvers and can be made explicit using the
representation introduced in Section~\ref{sec:local_representability}. 
For each subdomain \(i\in\mathcal I\), we define the local surrogate error
\[
\mathfrak{e}_i(\widehat{\mu})
:=
\mathfrak{s}_i(\widehat{\mu})
-
\widetilde{\mathfrak{s}}_i(\widehat{\mu})
\in\mathcal V(\Omega_i).
\]
By local representability, the same error can be expressed on the reference domain as
\[
\mathfrak{e}_{\star}(\widehat{\mu})
:=
\mathfrak{s}_{\star}(\widehat{\mu})
-
\widetilde{\mathfrak{s}}_{\star}(\widehat{\mu})
\in\mathcal V(\Omega_{\star}).
\]

The stability of the DD algorithm under the perturbations induced by the surrogate models is based on the following assumptions. We first assume a uniform approximation bound on the reference surrogate.

\begin{asm}[Uniform approximation property]
    The reference surrogate satisfies
\[
\sup_{\widehat{\mu}\in\widehat{\mathcal P}}
\left\|
\mathfrak{e}_{\star}(\widehat{\mu})
\right\|_{\mathcal V(\Omega_{\star})}
\leq \varepsilon.
\]
\end{asm}

Moreover, we assume that the composition of the local solution maps with the interface exchange and global assembly operations is stable with respect to local perturbations. 

\begin{asm}[Stability of the DD iterations]
\label{asm:cdd}
There exists a constant
\(C_{\mathrm{DD}}>0\) such that
\[
\left\|\delta^{(k)}\right\|_{\star}
\leq
C_{\mathrm{DD}}
\max_{i\in\mathcal I}
\left\|
\mathfrak{e}_{\star}
\left(\widehat{\mu}_i^{(k)}\right)
\right\|_{\mathcal V(\Omega_{\star})}.
% \implies
% \left\|\delta^{(k)}\right\|_{\star}
% \leq
% C_{\mathrm{DD}}\varepsilon,
% \ k\geq0.
\]
\end{asm}
The assumption \ref{asm:cdd} is a Lipschitz continuity assumption of the DD operator with respect to perturbations of the local solvers. In optimized Schwarz methods \cite{gander2012best}, global iterations are obtained through the composition of a bounded trace, Robin transmission, restriction, prolongation, and coarse-assembly operators. Consequently, perturbations of local solution maps propagate continuously to the global iteration, giving an estimate as in the assumption \ref{asm:cdd}, where \(C_{\mathrm{DD}}\) depends only on the stability constants of the underlying domain decomposition operators.

We can now quantify the effect of the surrogate perturbations on the DD
iteration.
\begin{prop}[Stability under surrogate perturbations]
\label{prop:perturbed_dd}
Let \(\mathcal{DD}^{\mu}\) be a contraction with constant \(q<1\) in the norm
\(\|\cdot\|_{\star}\), and assume that the perturbations induced by the
surrogate local solvers satisfy Assumption \ref{asm:cdd}.
Then, the surrogate-perturbed DD iterates satisfy
\[
\limsup_{k\to\infty}
\left\|
u^\mu-\widetilde{u}^{(k)}
\right\|_{\star}
\leq
\frac{C_{\mathrm{DD}}}{1-q}\varepsilon.
\]
\end{prop}

The proof of this property is reported  in the Appendix, where a more formal version of Proposition \ref{prop:appendix_perturbed_contraction} is formulated.
Proposition~\ref{prop:perturbed_dd} establishes a direct quantitative relation
between the approximation accuracy of the local surrogate and the global
accuracy of the resulting DD iteration. The exact contraction mechanism is
preserved up to an error floor proportional to the surrogate approximation
error. The amplification factor depends on two distinct properties of the
algorithm: the contraction factor \(q\), which determines how perturbations
accumulate across iterations, and the stability constant \(C_{\mathrm{DD}}\),
which measures how local surrogate errors are transferred through the
interface exchange and global assembly operations.

This observation has a direct implication for the design of the training
strategy. Under the assumptions above, uniform accuracy of the reference surrogate over the localized
parameter space \(\widehat{\mathcal P}\) is sufficient to control the
perturbation of the global iterative solver. Hence, the training set must
represent not only the variability of the local physical configurations, but
also the algorithmic states \(\widehat{\mu}_i^{(k)}\) generated along the DD
trajectories. This requirement motivates the cascaded training strategy
introduced in Section~4.2.

\section{A domain decomposition framework for mixed-dimensional problems}
\label{sec:3D1D_problem_localization}

We now instantiate the abstract framework of Section~\ref{sec:MultiScaleAbstract} on a mixed-dimensional elliptic problem coupling a three-dimensional bulk domain with an embedded one-dimensional structure.
We address a coupled problem defined by an exterior domain $\Omega\subset\mathbb{R}^3$ and an interior domain $\Sigma\subset\Omega$. We assume $\Sigma$ to be a generalized cylinder with a centerline $\Lambda$, the latter being a 1D domain with arc length parameter $s\in(0,S)$. An analogous formulation holds when the low-dimensional structure is given by the union of multiple substructures $\Sigma=\cup_i\Sigma_i$ and $\Lambda=\cup_i\Lambda_i$, each being a generalized cylinder. In this case, $\Lambda$ becomes a 1D graph embedded within a 3D domain. We assume that each cylinder has a constant radius $\varepsilon>0$, which is small compared to the diameter of $\Omega$. 

The dimensionality reduction process consists of replacing $\Sigma$ with $\Lambda$ while keeping some information about its original 3D structure and its interaction with the external domain $\Omega$. Let $\Upsilon:=\partial\Sigma$ be the interface of the coupled problem in the full 3D representation and let $\mathcal{T}_\Lambda$ be the restriction operator from $\Omega$ to $\Lambda$. In \cite{Laurino20192047,Kuchta2021a,Heltai20232425} it is defined as the composition of the trace operator on \(\Upsilon\) combined with a projection operator from \(\Upsilon\) to \(\Lambda\) based on cross-sectional averages.
We note that the averaging procedure plays a regularizing role: the ill-posed trace of a 3D function onto a 1D manifold (which is not defined in the usual Sobolev spaces) is replaced by the well-defined operator $\mathcal{T}_\Lambda$.
Then, the continuous 3D-1D coupled problem can be formulated as follows,
% \begin{equation}\label{eq:3D-1D-problem}
% \begin{cases}
% \nabla \cdot \left(-k_\Omega\nabla u\right)+2\pi\epsilon\left(\mathcal{T}_{\Lambda} u-u_{\Lambda}\right)\delta_{\Lambda}=0, & \text{on}\; \Omega,\\
% d_s \left(-\pi \epsilon^2 k_{\Lambda}d_s u_{\Lambda}\right)+2\pi\epsilon\left(u_{\Lambda}-\mathcal{T}_{\Lambda} u \right)=0, & \text{on}\; \Lambda, \\ 
% -k_{\partial \Omega}\nabla u \cdot \mathbf{n}+\beta(u-g_{\partial\Omega})=0, & \text{on}\; \partial \Omega, \\
% u_{\Lambda}=g_{\partial\Lambda}, & \text{on} \; \partial \Lambda_D.
% \end{cases}
% \end{equation}  
\begin{equation}\label{eq:general-problem}
\begin{cases}
K^\mu_\Omega(u) + M^\mu_{\Omega\Omega}(u) - M^\mu_{\Omega\Lambda}(u_\Lambda) = 0 & \text{in } \Omega, \\
K^\mu_\Lambda(u_\Lambda) + M^\mu_{\Lambda\Lambda}(u_\Lambda) - M^\mu_{\Lambda\Omega}(u) = 0 & \text{in } \Lambda, \\
B^\mu_\Omega(u) = g^\mu_\Omega & \text{on } \partial \Omega, \\
B^\mu_\Lambda(\mu)(u_\Lambda) = g^\mu_\Lambda & \text{on } \partial \Lambda,
\end{cases}
\end{equation}
where:
\begin{align*}
    & K^\mu_\Omega(u):=\nabla \cdot \left(-k_\Omega\nabla u\right); \ 
    M^\mu_{\Omega\Omega}(u):=2\pi\epsilon\left(\mathcal{T}_{\Lambda} u\right)\delta_{\Lambda}, \
    M^\mu_{\Omega\Lambda}(u_\Lambda):=2\pi\epsilon\left(u_{\Lambda}\right)\delta_{\Lambda}; 
    \\ 
    & K^\mu_\Lambda(u_\Lambda):=d_s \left(-\pi \epsilon^2 k_{\Lambda}d_s u_{\Lambda}\right); \ 
    M^\mu_{\Lambda\Lambda}(u_\Lambda):=2\pi\epsilon\left(u_{\Lambda} \right); \
    M^\mu_{\Lambda\Omega}(u):=2\pi\epsilon\left(\mathcal{T}_{\Lambda} u \right); 
    \\ 
    & B^\mu_\Omega(u) :=-k_{\partial \Omega}\nabla u \cdot \mathbf{n}+\beta u; \ 
    B^\mu_\Lambda(\mu)(u_\Lambda):=u_{\Lambda}|_{\partial\Lambda}.
\end{align*}
For notational convenience and for coherence with the abstract problem formulation \eqref{eq:split-param}, we denote the parametrization of the 1D graph $\Lambda$ as $\mu$.
We note that problem \eqref{eq:general-problem} contains some abuse of notation as $K^\mu_\Omega(u)$ does not actually depend on the parameter $\mu$, while the notation $K^\mu_\Lambda(u_\Lambda)$ is redundant since $\mu$ and $\Lambda$ both refer to the 1D graph.

The functions $u, u_\Lambda$ are coupled unknowns in the 3D and 1D domains, respectively, and $k_\Omega, \sigma_\Omega , k_\Lambda$ represent the coefficients of the elliptic operators.  
We note that the first equation of \eqref{eq:general-problem} is only formally defined in its strong form, as the symbol $\delta_\Lambda$ denotes the Dirac distribution supported on the one-dimensional manifold $\Lambda \subset \Omega$. The coupling term $2\pi\epsilon\left(\mathcal{T}_{\Lambda} u-u_{\Lambda}\right)\delta_{\Lambda}$ is obtained by averaging on $\Upsilon$ the coupling condition $- k_\Omega \nabla u \cdot \mathbf{n} = u - u_\Lambda $, similar to the one adopted in \cite{Laurino20192047}. In fact, the coefficients $2\pi\epsilon$ and $\pi\epsilon^2$ represent the cross-sectional measures of $\Upsilon$ and $\Sigma$, respectively.

We are interested in solving multiple instances of this problem for different configurations of the low-dimensional structure, including, for example, changes in the topology or density of the 1D graph $\Lambda$. To formalize this fact, it is useful to define a suitable \textit{parameter space} $\mathcal{P}$, which collects all configurations of interest. There are many ways to define this parameterization, the choice depending on the surrogate model architecture, as will be discussed later on.

The first step toward an efficient solution is to reformulate the coupled problem \eqref{eq:general-problem} into a \emph{weakly coupled algorithm}, where the 3D and 1D problems are solved separately.
This algorithm reads as follows: given an initial guess $u_\Lambda^0$, for $k=1,2\ldots$ solve sequentially the following two sub-problems until a suitable convergence criterion is satisfied:

\begin{equation}\label{eq:3d_subprob}
    \mathfrak{p}_\Omega^\mu(u^{(k)}) :=
    \begin{cases}
      {K}^\mu_{\Omega}(u^{(k)}) + {M}^\mu_{\Omega\Omega}(u^{(k)}) = {M}^\mu_{\Omega\Lambda}(u_\Lambda^{(k-1)}) & \text{ in } \Omega, \\[4pt]
      B^\mu_{\Omega}(u^{(k)}) = g_\Omega & \text{ on } \partial \Omega ,
    \end{cases}
\end{equation} 
and
\begin{equation}\label{glob_1d}
    \mathfrak{p}^\mu_\Lambda(u_\Lambda^{(k)}) :=
    \begin{cases}
      K^\mu_{\Lambda}(u_\Lambda^{(k)}) + M_{\Lambda\Lambda}(u_\Lambda^{(k)}) = M_{\Lambda\Omega}(u^{(k)}) & \text{ in } \Lambda, \\[4pt]
      B_{\Lambda}(u_\Lambda^{(k)}) = g_\Lambda & \text{ on } \partial \Lambda .
    \end{cases}
\end{equation}

In the present work, the DD reduction is applied only to the three-dimensional bulk problem \(\mathfrak{p}_\Omega^\mu\), which is the computationally dominant component of the coupled system. The one-dimensional problem \(\mathfrak{p}_\Lambda^\mu\) is retained as a global solve within the staggered iteration. For simplicity of notation, we omit the subscript $\Omega$, since it is now understood without ambiguity that $\Omega$ is the reference domain. In the following section, we detail the formulation of this solver, explaining how the global problem is decomposed, how the local problems are specified, and how the interface conditions ensure consistency between subdomains.

% \begin{figure}[h!]
%     \centering
%     \includegraphics[width=0.85\linewidth]{Figures/algo_scheme.pdf}
%     \caption{Schematization of the weakly coupled algorithm.
%     \textbf{Top:} conceptual representation of the decoupling step between the coupled 3D--1D problem, illustrating the separation of the 3D domain $\Omega$ and the 1D network $\Lambda$. The dotted arrows refer to staggered (iterative) solution strategy. 
%     \textbf{Bottom:} schematic view of the domain decomposition procedure applied to the 3D subproblem within the decoupled staggered framework.}
%     \label{fig:dd-scheme}
% \end{figure}

\subsection{A non-overlapping domain decomposition algorithm}

We now proceed to a detailed presentation of the domain decomposition solver. 
We adopt a non-overlapping domain decomposition formulation, although the localized learning framework is not intrinsically restricted to this setting. Overlapping formulations provide a natural alternative and their comparison with the present approach is left for future investigation. Concerning the transmission conditions, we employ Robin conditions on all internal interfaces, leading to a Robin-Robin (RR) optimized Schwarz method. This choice provides a uniform interface treatment, in which the same type of transmission operator acts on every subdomain interface, and is therefore particularly well suited to the construction of reusable local surrogate models.

We assume that \(\Omega\subset\mathbb{R}^3\) is a tensor-product domain and consider a partition into congruent Cartesian subdomains,
\(\overline{\Omega}
= \cup_{i\in\mathcal I}\overline{\Omega}_i,
\ \Omega_i\cap \Omega_j=\emptyset
\ \text{for }i\neq j\).
The interface between two subdomains $\Omega_i$ and $\Omega_j$ is written as $\Gamma_{ij} := \partial \Omega_i \cap \partial \Omega_j$. For any quantity $x$ defined on $\Omega$, its restriction to $\Omega_i$ is denoted by $x|_{\Omega_i} =: x_i$.

To simplify the exposition, let us focus on a pair of neighboring subdomains $(\Omega_i, \Omega_j)$. Given a global solution $u$ defined on $\Omega$, the following continuity conditions must hold across the common interface $\Gamma_{ij}$:
\begin{equation}\label{eq:dd-cont}
    u_{i}|_{\Gamma_{ij}} = u_{j}|_{\Gamma_{ij}}, \quad
    k_\Omega \nabla u_i \cdot \mathbf{n}_i = - k_\Omega \nabla u_j \cdot \mathbf{n}_j \quad \text{on } \Gamma_{ij},
\end{equation}
where $\mathbf{n}_{i}$ and $\mathbf{n}_{j}$ denote the outward unit normal vectors to $\partial \Omega_i$ and $\partial \Omega_j$, respectively, and for notational simplicity $k_\Omega$ is assumed to be constant.
Let $K^\mu_{\Omega,i},\,M^\mu_{\Omega\Omega,i},\,M^\mu_{\Omega\Lambda,i}$ be the restrictions of the operators of problem \eqref{eq:3d_subprob} on $\Omega_i$ and for notational convenience, we set
\begin{equation}
\label{eq:A_definition}
A^\mu_i := K^\mu_{\Omega,i} + M_{\Omega\Omega,i}, 
\quad 
b^\mu_i := \,M^\mu_{\Omega\Lambda,i}(u_{\Lambda,i}),
\end{equation}
so that the $i$-th local subproblem reads
\begin{equation}\label{eq:gen_loc_i}
    \mathfrak{p}_i^\mu(u_i) :=
    \begin{cases}
    A^\mu_i(u_i) = b_i & \text{in } \Omega_i, \\
    B_{\Omega_i}(u_i) = g_i & \text{on } \partial \Omega \cap \partial \Omega_i, \\
    - k_\Omega \nabla u_i \cdot \mathbf{n}_i + \rho\, u_i|_{\Gamma_{ij}}
    = - k_\Omega  \nabla u_j \cdot \mathbf{n}_j + \rho\, u_j|_{\Gamma_{ij}}
    & \text{on } \Gamma_{ij},
    \end{cases}
\end{equation}
where $\rho > 0$ is an acceleration parameter and it can be tuned to optimize the convergence rate of the algorithm \cite{gander2012best}.

 As already noted, the formulation and parametrization of the general local problem \eqref{eq:gen_loc_i} differ from those of the global problem. Therefore, we denote the collection of local problems by \(\{\mathfrak{p}_{i}^{\hat{\mu}}\}_{i\in \mathcal{I}},\) where \(\hat{\mu} \in \widehat{\mathcal{P}}\) belongs to a modified parameter space with respect to the global one. This space accounts for a larger variability at the microscale and for different configurations of the boundary conditions, so as to represent internal boundaries between subdomains. Consequently, we can recall the notion of \textit{local representability} applied to the mixed-dimensional 3D subproblem: for each \textit{global} 3D problem \(\mathfrak{p}_\Omega^\mu(u)\) defined in \eqref{eq:3d_subprob} there exists a parameter \(\hat{\mu} \in \hat{\mathcal{P}}\) such that the collection of discrete Robin subproblems \(\{\mathfrak{p}_i^{\hat{\mu}}(u_i)\}_{i \in \mathcal{I}}\) defined by \eqref{eq:gen_loc_i} is equivalent to \(\mathfrak{p}_\Omega^\mu(u)\).

\subsection{Numerical discretization by the finite element method}
\label{sec:fem_discretization}

The domain \(\Omega\) is equipped with a uniform conforming mesh
\(\mathcal T_h^\Omega\), aligned with the subdomain interfaces. Hence, the
restriction of \(\mathcal T_h^\Omega\) to each \(\Omega_i\) defines a local
mesh \(\mathcal T_{h,i}\), and the local meshes match across every interface
\(\Gamma_{ij}:=\partial\Omega_i\cap\partial\Omega_j\).
%Each subdomain is the affine image of the unit reference domain
%\(\Omega_\star\) through a bijective affine map
%\(\Phi_i:\Omega_\star\to\Omega_i\).
The uniform partition and mesh are chosen such that
\(\Phi_i^{-1}(\mathcal T_{h,i})
=\mathcal T_{h,\star},
\ \forall i\in\mathcal I,\)
where \(\mathcal T_{h,\star}\) is a common reference mesh on
\(\Omega_\star\). 
In the Cartesian decompositions considered here, $\Phi_i$ is simply an affine transformation. Thus, all local meshes and finite element spaces are affine-equivalent to the same reference discretization.

Let \(\mathcal V_{h,i} \subset H^1(\Omega_i)\)
denote the local finite element space. 
%The pullback
%\(\mathcal T_i:V_{h,i}\to V_{h,\star} \ %\mathcal T_i v_h:=v_h\circ\Phi_i\)
%identifies every local discrete space with the common reference space
%\(V_{h,\star}\). 
We then introduce the broken finite element space
\(\mathcal{V}_h := \mathcal{V}_{h,i} \oplus \mathcal{V}_{h,j}\).

The weak formulation of the problem reads as follows: find $u_h=(u_{h,i}, u_{h,j}) \in \mathcal{V}_h$ such that 
\begin{equation*}
\begin{cases}
\begin{aligned}   
    &a_i(u_{h,i}, v_{h,i})-(k_\Omega \nabla u_{h,i} \cdot { n_i}, v_{h,i})_{\Gamma_{ij}} = b_i(v_{h,i}), \quad \text{for all} \; v_{h,i} \in \mathcal{V}_h(\Omega_i)\\   
    &(k_\Omega \nabla u_{h,i} \cdot {\mathbf{n}_i}, v_{h,i})_{\Gamma_{ij}} +\rho (u_{h,i}, v_{h,i})_{\Gamma_{ij}}  = 
    -(k_\Omega\nabla u_{h,j} \cdot {\mathbf{n}_j}, v_{h,i})_{\Gamma_{ij}} +\rho(u_{h,j}, v_{h,i})_{\Gamma_{ij}} 
\end{aligned} 
\end{cases}
\end{equation*} 
and 
\begin{equation*}
\begin{cases}
\begin{aligned}   
    & a_j(u_{h,j}, v_{h,j})-(k_\Omega\nabla u_{h,j} \cdot { n_j}, v_{h,j})_{\Gamma_{ij}} = b_j(v_{h,j}) \quad \text{for all} \; v_{h,j} \in \mathcal{V}_h(\Omega_j)\\   
    & (k_\Omega\nabla u_{h,j} \cdot {\mathbf{n}_j}, v_{h,j})_{\Gamma_{ij}} +\rho (u_{h,j}, v_{h,j})_{\Gamma_{ij}}  = 
    -(k_\Omega\nabla u_{h,i} \cdot {\mathbf{n}_i}, v_{h,j})_{\Gamma_{ij}} +\rho(u_{h,i}, v_{h,j})_{\Gamma_{ij}}  
\end{aligned} 
\end{cases}
\end{equation*}
where
\begin{equation*}
    a_i(u_{h,i}, v_{h,i}):= (k_\Omega \nabla u_{h,i}, \nabla v_{h,i})_{\Omega_i} + 2\pi \epsilon (\overline{\mathcal{T}}_{\Lambda_i} u_{h,i}, \overline{\mathcal{T}}_{\Lambda_i} v_{h,i})_{\Lambda_i} ,
\end{equation*}
and 
\begin{equation*}
    b_i(v_{h,i}):= 2\pi \epsilon( u_{\Lambda,i}, \overline{\mathcal{T}}_{\Lambda_i} v_{h,i})_{\Lambda_i}.
\end{equation*}

Let $\mathcal{L}_i $ denote the lifting operator from $\mathcal{V}_{h,j}$ to $\mathcal{V}_{h,i}$, defined so that
\((\mathcal{L}_i v_{h,j})|_{\Gamma_{ij}} = v_{h,j}|_{\Gamma_{ij}}\).
%This operator is defined to map the trace of a basis function from $\mathcal{V}_{h,j}$ to the corresponding basis function in $\mathcal{V}_{h,i}$, satisfying the following:
%\[(\mathcal{L}_i \varphi_{h,j})|_{\Gamma_{ij}} = \varphi_{h,j}|_{\Gamma_{ij}}, \quad \mathcal{L}_i \varphi_{h,j}=\varphi_{h,i}.\]
The operator $\mathcal{L}_i$ establishes a correspondence between the basis functions of the two subdomains that share the same degrees of freedom along the interface.
Since $(\cdot,\cdot)_{\Gamma_{ij}}$ depends only on traces, we have
\[(k_\Omega \nabla u_{h,i}\!\cdot \mathbf{n}_i,\, v_{h,j})_{\Gamma_{ij}}
= (k_\Omega \nabla u_{h,i}\!\cdot \mathbf{n}_i,\, \mathcal{L}_i v_{h,j})_{\Gamma_{ij}},\]
therefore,
\[(k_\Omega \nabla u_{h,i}\!\cdot \mathbf{n}_i,\, v_{h,j})_{\Gamma_{ij}}
= a_i\!\big(u_{h,i},\, \mathcal{L}_i v_{h,j}\big)-\,b_i(\mathcal{L}_i v_{h,j}).\]
An analogous relation holds for the flux term $(\nabla u_{h,j}\cdot \mathbf{n}_j, v_{h,i})_{\Gamma_{ij}}$.

The continuity of Robin traces implies the subspace equation on $\Omega_i$ 
\begin{equation}
\label{eq:var_i}
    a_i(u_{h,i}, v_{h,i}) +\rho(u_{h,i}, v_{h,i})_\Gamma  = b_i(v_{h,i}) +  b_j(\mathcal{L}_j\,v_{h,i})-a_j(u_{h,j}, \mathcal{L}_j\,v_{h,i}) + \rho(u_{h,j}, \mathcal{L}_j\,v_{h,i})_\Gamma ;
\end{equation}
and similarly, for subdomain $\Omega_j$,
\begin{equation}
\label{eq:var_j}
    a_j(u_{h,j}, v_{h,j}) +\rho(u_{h,j}, v_{h,j})_\Gamma  = b_j(v_{h,j}) +  b_i(\mathcal{L}_i\,v_{h,j})-a_i(u_{h,i}, \mathcal{L}_i\,v_{h,j}) + \rho(u_{h,i}, \mathcal{L}_i\,v_{h,j})_\Gamma .
\end{equation}
Here, we refer to
\[f_{h,j}(v_{h,i}) := b_j(\mathcal{L}_j\,v_{h,i})-a_j(u_{h,j}, \mathcal{L}_j\,v_{h,i}) + \rho(u_{h,j}, \mathcal{L}_j\,v_{h,i})_{\Gamma_{ij}} \]
as the discrete Robin transmission data from domain $\Omega_j$ to domain $\Omega_i$ (the definition is entirely analogous to that of $f_j$).
As a result, the variational form of the local problem in $\Omega_i$ reads as follows,
\begin{equation}
\label{eq:gen_var_i}
    a_i(u_{h,i}, v_{h,i}) +\rho(u_{h,i}, v_{h,i})_\Gamma  = b_i(v_{h,i}) +    f_{h,j}(v_{h,i}) \quad \forall v_{h,i} \in \mathcal{V}_{h,i}.
\end{equation}

To define the iterative domain decomposition algorithm, we adopt a block-Jacobi
strategy. Let \(k\geq 0\) denote the iteration index. Given the previous iterate
\(
u_h^{(k)} \in \mathcal V_{h,i} \oplus \mathcal V_{h,j},
\)
the new iterate
\(
u_h^{(k+1)} := u_{h,i}^{(k+1)} \oplus u_{h,j}^{(k+1)}
\)
is obtained by solving, in parallel, the local subproblems 
\begin{equation}
\label{eq:fucitonalDD}
\begin{cases}
    \vspace{10pt} a_i(u^{(k+1)}_{h,i}, v_{h,i}) +\rho(u_{h,i}^{(k+1)}, v_{h,i})_\Gamma  = b_i(v_{h,i}) +  f^{(k)}_{h,j}(v_{h,i}) ;\\[10pt]
    a_j(u^{(k+1)}_{h,j}, v_{h,j}) +\rho(u^{(k+1)}_{h,j}, v_{h,j})_\Gamma  = b_j(v_{h,j}) + f_{h,i}^{(k)}(v_{h,j}) .\vspace{10pt}
\end{cases}
\end{equation}
where the incoming Robin data from \(\Omega_j\) at iteration \(k\) are defined by
\[
f_{h,j}^{(k)}(v_{h,i})
:=
b_j(\mathcal{L}_j\,v_{h,i})
-
a_j(u_{h,j}^{(k)}, \mathcal{L}_j\,v_{h,i})
+
\rho(u_{h,j}^{(k)}, \mathcal{L}_j\,v_{h,i})_{\Gamma_{ij}},
\]
and \(f_{h,i}^{(k)}\) is defined analogously by exchanging the roles of \(i\) and \(j\).

Under standard assumptions on the finite element discretization, the non-overlapping Robin--Robin iteration satisfies the contraction property assumed in Section~\ref{sec:local_perturbation_theory}. More precisely, let \(\mathcal{W}_h \subset H^1(\Omega)\) denote the global conforming finite element space associated with \(\mathcal{T}_h^\Omega\), whose functions are continuous across the subdomain interfaces, and let \(u_h\in\mathcal{W}_h\) be the solution of the direct finite element discretization of the three-dimensional problem~\eqref{eq:general-problem}. The following proposition holds.

\begin{prop}[Contraction of the discrete Robin--Robin iteration]
Assume that the hypotheses in the detailed formulation of Proposition~\ref{prop:exactRR-OSMconvergence} reported in the Appendix are satisfied.
Let \(u_{h,i}^{(k)}\in \mathcal V_{h,i}\) denote the local Robin--Robin iterate on
\(\Omega_i\) at iteration \(k\geq 1\), and define
\(
e_{h,i}^{(k)} := u_{h,i}^{(k)} - u_h|_{\Omega_i}.
\)
Moreover, having set \(\|v\|_{a,i}^2:=a_i(v,v)\), let
\[
\|e_h\|_{\star}
:=
\left(
\sum_{i\in\mathcal{I}}
\|e_{h,i}\|_{a,i}^2
+
\sum_{i\in\mathcal{I}}
\sum_{j\in\mathcal{N}(i)}
\tfrac{1}{2}\,\tfrac{\gamma}{h}\,
\|e_{h,i}\|_{L^2(\Gamma_{ij})}^2
\right)^{1/2}.
\] 
Then the Robin--Robin iteration is a contraction in the norm \(\|\cdot\|_{\star}\).
More precisely, there exists a constant \(q\in(0,1)\) such that, for every \(k\geq 1\),
\[
\|e_h^{(k)}\|_{\star}
\leq
q\,
\|e_h^{(k-1)}\|_{\star}.
\]
\end{prop}

The proof of this result is reported in the Appendix, where a more formal version of Proposition \ref{prop:exactRR-OSMconvergence} is available.
Hence, the discrete DD iteration provides a concrete realization of the contractive map introduced in Section~\ref{sec:local_perturbation_theory}, with \(u_h\) as its unique fixed point. The perturbation analysis developed there therefore applies directly when the exact local finite element solves are replaced by surrogate approximations.

\medskip

\subsection{Matrix Formulation of the domain decomposition algorithm}
Now, let $\mathsf{A_i}$ be the matrix representation of the bilinear form $a_i(\varphi_{h,i},\varphi_{h,i})$, and let $\mathsf{b_i} := b_i(\varphi_{h,i})_{\Omega_i}$ 
be the corresponding load vector with $v_{h,i} \in \mathcal{V}_{h,i}$. 
%Recall that, given the one-dimensional data $u_{\Lambda,h}$, the forcing term is defined as  \(b_i = M_{\Omega\Lambda}(u_{\Lambda,h})|_{\Omega_i}\), we introduce $\mathsf{M_i}$ as the finite element discretization of the coupling operator $M_{\Omega\Lambda}$.  Hence, the discrete forcing vector can be expressed as \(\mathsf{b_i} = \mathsf{M_i}\, \mathsf{u}_{\Lambda,\mathsf{h}}\), where $\mathsf{u}_{\Lambda,\mathsf{h}}$ denotes the discrete representation of the one-dimensional data $u_{\Lambda,h}$.
Analogously, we define $\mathsf{A_j}$ and $\mathsf{b_j}$, and let $\mathsf{u_{i}}$ and $\mathsf{u_{j}}$ denote the vectors of solution coefficients associated with $u_{h,i}$ and $u_{h,j}$, respectively. Finally, let $\mathsf{G}^{\Gamma_{ij}}$ denote the matrix representation of the inner product $(\cdot,\,\cdot)_{\Gamma_{ij}}$ along the virtual interface $\Gamma_{ij}$. Note that from now on we use the typesetting style $\mathsf A, \mathsf v$ to denote algebraic operators and the vector representation of discrete functions.

Consider a partition of local basis functions into two sets: the basis
functions associated with the interior degrees of freedom of $\Omega_i$,
indexed by $I$, and those associated with the degrees of freedom lying on
the internal interfaces of $\partial\Omega_i$, indexed by $\Gamma$. Here,
$\Gamma$ denotes the union of possibly multiple interfaces shared by
$\Omega_i$ with its neighboring subdomains. This partition induces the
corresponding block decomposition of the local matrices and vectors.

At the algebraic level, the local Robin problem can then be written as
\begin{equation}\label{eq:gen_alg_i}
\mathfrak{p}_i^\mu(\mathsf{u}_i) :=
 \left[
    \begin{array}{cc}
        \mathsf{A}_{\mathsf{i}}^{II} & \mathsf{A}_{\mathsf{i}}^{I\Gamma} \\[7pt]
        \mathsf{A}_{\mathsf{i}}^{\Gamma I} & \mathsf{A}_{\mathsf{i}}^{\Gamma \Gamma} + \rho \, \mathsf{G}_\mathsf{i}^{\Gamma}
    \end{array} \right]
    \left[\begin{array}{c}
         \mathsf{u}_{\mathsf{i}}^{I}  \\[7pt]
         \mathsf{u}_{\mathsf{i}}^{\Gamma} 
    \end{array}\right]
    =
    \left[\begin{array}{c}
         \mathsf{b}_{\mathsf{i}}^{I}  \\[7pt]
         \mathsf{b}_{\mathsf{i}}^{\Gamma} +  \sum_{j \neq i}\mathsf{L}_{\mathsf{i}} \, \mathsf{L}_{\mathsf{j}}^\top \, \mathsf{f}_{\mathsf{j}}
    \end{array}\right] ,
\end{equation}
where
\(\mathsf{f}_{\mathsf{j}}
:= \mathsf{b}_{\mathsf{j}}
- \mathsf{A}_{\mathsf{j}} \, \mathsf{u}_{\mathsf{j}}
 + \rho \mathsf{G}_{\mathsf{j}}^{\Gamma_{ij}} \, \mathsf{u}_{\mathsf{j}}^{\Gamma_{j}}\)
is the discrete Robin transmission vector generated by subdomain
$\Omega_j$. The sum in~\eqref{eq:gen_alg_i} collects the Robin
data transferred from all neighboring subdomains to $\Omega_i$.

Due to the assumptions of uniformity and conformity in the computational mesh, the \textit{restriction matrix} $\mathsf{L_i}$ maps the global degrees of freedom of $\mathcal{T}_h^{\Omega}$ to the local degrees of freedom of $\mathcal{T}_h^{\Omega_i}$. It is a Boolean matrix of size $N_i \times N$, with entries \(\mathsf{L_i}(m,n) = 1\) if the $m$-th local and $n$-th global degrees of freedom coincide, and \(\mathsf{L_i}(m,n) = 0\) otherwise. Thus, $\mathsf{L_i}$ restricts global coefficient vectors to $\Omega_i$, while $\mathsf{L_i}^\top$ extends local vectors to the global space by zero-padding. Consequently, $\mathsf{L_i}\mathsf{L_j}^\top$ transfers the degrees of freedom shared by $\Omega_j$ and $\Omega_i$ and represents the discrete counterpart of the continuous lifting operator $\mathcal{L}_i$. The discrete Robin trace transferred from $\Omega_j$ to $\Omega_i$ is therefore given by $\mathsf{L_i}\mathsf{L_j}^\top \mathsf{f_j}$.

To define the iterative DD solver,
%we adopt a block-Jacobi strategy, in
%which all local subproblems are updated using Robin data computed from the
%previous iteration.
%Accordingly, for $k\geq 1$, the discrete local problem
%on $\Omega_i$ is
accordingly to equation~\ref{eq:fucitonalDD}, for $k\geq 1$, the discrete local problem on $\Omega_i$ reads

\begin{equation}\label{eq:gen_alg_ik}
\mathfrak{p}_i^\mu(\mathsf{u}_i^{(k)}) :=
 \left[
    \begin{array}{cc}
        \mathsf{A}_{\mathsf{i}}^{II} & \mathsf{A}_{\mathsf{i}}^{I\Gamma} \\[7pt]
        \mathsf{A}_{\mathsf{i}}^{\Gamma I} & \mathsf{A}_{\mathsf{i}}^{\Gamma \Gamma} + \rho \, \mathsf{G}_\mathsf{i}^{\Gamma}
    \end{array} \right]
    \left[\begin{array}{c}
         \mathsf{u}_{\mathsf{i}}^{I, (k)}  \\[7pt]
         \mathsf{u}_{\mathsf{i}}^{\Gamma, (k)} 
    \end{array}\right]
    =
    \left[\begin{array}{c}
         \mathsf{b}_{\mathsf{i}}^{I}  \\[7pt]
         \mathsf{b}_{\mathsf{i}}^{\Gamma} +  \sum_{j \neq i}\mathsf{L}_{\mathsf{i}} \, \mathsf{L}_{\mathsf{j}}^\top \, \mathsf{f}_{\mathsf{j}}^{(k-1)}
    \end{array}\right].
\end{equation}

This choice yields mutually independent subdomain solves at each DD iteration and is therefore naturally parallelizable. A block Gauss--Seidel strategy, in which the most recently updated Robin data are used as soon as they become available, could be adopted within the same formulation.

\subsection{A two-level domain decomposition method}

The convergence of one-level DD methods generally deteriorates as the number
of subdomains increases, since global error components are not efficiently
reduced by local iterations. We therefore introduce a coarse correction in
the conforming finite element space
\(\mathcal{V}_c\subset\mathcal W_h\subset \mathcal{V}_h\).
Let \(\mathcal T_H^\Omega\) be a coarse triangulation conforming with the
subdomain partition and nested in \(\mathcal T_h^\Omega\). We define
\[
\mathcal{V}_c
:=
\left\{
v\in C^0(\overline\Omega):
v|_K\in\mathbb P_1(K),
\quad
K\in\mathcal T_H^\Omega
\right\}.
\]
More sophisticated coarse spaces could be considered \cite{klawonn2016comparison}, but the present
geometric coarse space is sufficient for the purpose of this work.

Let \(\mathsf P:\mathcal V_c\to\mathcal W_h\) denote the canonical coarse-to-fine prolongation. In algebraic form, the Galerkin coarse operator is
\(\mathsf A_c := \mathsf P^\top\mathsf A\mathsf P\).
Given the current iterate \(\mathsf u^{(k)}\), the coarse correction is
obtained from
\[
\mathsf A_c\mathsf e_c^{(k)}
=
\mathsf P^\top\mathsf r^{(k)},
\qquad
\mathsf r^{(k)}
:=
\mathsf b-\mathsf A\mathsf u^{(k)},
\]
and the corrected iterate is
\(\mathsf u_{\rm corr}^{(k)} = \mathsf u^{(k)} + \mathsf P\mathsf e_c^{(k)}\).
Thus, the coarse solve removes the component of the current error represented
in the prolonged coarse space.

Note that the global notation above is only expository. Neither the global fine matrix \(\mathsf A\) nor the global residual \(\mathsf r^{(k)}\) needs to be assembled. The coarse matrix \(\mathsf A_c\) is assembled directly on
\(\mathcal{V}_c\), while the residual is evaluated from local contributions. If \(\mathsf P_i\) denotes the local coarse-to-fine prolongation on \(\Omega_i\) and the diagonal weighting matrices \(\mathsf D_i\) satisfy
\[
\sum_{i\in\mathcal I}\mathsf L_i^\top\mathsf D_i\mathsf L_i=\mathsf I,
\qquad\text{then}\qquad
\mathsf P^\top\mathsf r^{(k)}
=
\sum_{i\in\mathcal I}\mathsf P_i^\top\mathsf D_i\mathsf r_i^{(k)}.
\]
Hence, the coarse residual is assembled entirely from local residual
contributions and local fine-to-coarse transfers, without assembling the
global fine-scale operator. This local decomposition of the coarse residual
will provide the basis for the coarse surrogate introduced in the following.

\subsection{The 3D-1D coupled domain decomposition algorithm}

The resulting domain decomposition algorithm combines the weakly coupled 3D--1D iteration with the two-level non-overlapping DD solver for the three-dimensional problem. The corresponding high-fidelity formulation, hereafter denoted as MS-DD-FOM, is summarized in Algorithm~\ref{alg:fullweaklyDD}.

\begin{algorithm}
\caption{Two-level DD for the coupled 3D--1D problem}
\label{alg:fullweaklyDD}
%\footnotesize
\begin{algorithmic}[1]

%\State \textbf{Input:} subdomain decompositions
%\(\{\Omega_i\}_{i\in\mathcal I}\), \(\{\Lambda_i\}_{i\in\mathcal I}\),
%restriction matrices \(\{\mathsf L_i\}_{i\in\mathcal I}\), weights
%\(\{\mathsf D_i\}_{i\in\mathcal I}\), local prolongations
%\(\{\mathsf P_i\}_{i\in\mathcal I}\), Robin parameter \(\rho\), boundary data
%\(g_\Lambda\)

%\Statex

%\State \textbf{Output:} discrete solutions
%\((\mathsf u^{(k)},\mathsf u_\Lambda^{(k)})\)

\Statex

\State Choose an initial guess
\((\mathsf u^{(0)},\mathsf u_\Lambda^{(0)})\)

\State Assemble the global 1D problem, 
\(\mathsf A_c\) and
\(\{\mathsf A_i,\mathsf G_i^\Gamma\}_{i\in\mathcal I}\)

\Statex

\For{\(k=1,2,\ldots\) until convergence}

    \Statex
    \Statex \(\triangleright\) \textbf{Coarse correction}

    \State Compute local residuals
    \(\{\mathsf r_i^{(k-1)}\}_{i\in\mathcal I}\)
    \Comment{local residual evaluation}

    \State \(\displaystyle
    \mathsf r_c^{(k-1)}
    \gets
    \sum_{i\in\mathcal I}
    \mathsf P_i^\top \mathsf D_i \mathsf r_i^{(k-1)}
    \)
    \Comment{coarse residual}

    \State \(\displaystyle
    \mathsf e_c^{(k-1)}
    \gets
    \mathsf A_c^{-1}\mathsf r_c^{(k-1)}
    \)
    \Comment{coarse solve}

    \State \(\displaystyle
    \mathsf u^{(k-1)}
    \gets
    \mathsf u^{(k-1)}+\mathsf P\mathsf e_c^{(k-1)}
    \)
    \Comment{coarse update}

    \Statex
    \Statex \(\triangleright\) \textbf{Local DD iteration}

    \For{\(i\in\mathcal I\)}

        \State \(\displaystyle
        \mathsf u_i^{(k-1)}
        \gets
        \mathsf L_i\mathsf u^{(k-1)}
        \)
        \Comment{restriction to \(\Omega_i\)}

        \State \(\displaystyle
        \mathsf f_i^{(k-1)}
        \gets
        \mathsf b_i^{(k-1)}
        -
        \mathsf A_i\mathsf u_i^{(k-1)}
        +
        \rho\,\mathsf G_i^\Gamma
        \mathsf u_i^{\Gamma_i,(k-1)}
        \)
        \Comment{Robin residual evaluation}

        \State \(\displaystyle
        \mathsf f_{\mathrm{agg},i}^{(k-1)}
        \gets
        \sum_{j\neq i}
        \mathsf L_i\mathsf L_j^\top
        \mathsf f_j^{(k-1)}
        \)
        \Comment{interface aggregation}

        \State Set
        \(\hat\mu_i^{(k-1)} := (\mu_i,\mathsf f_{\mathrm{agg},i}^{(k-1)})\)
        \Comment{localized parameter with Robin interface data}

        \State \(\displaystyle
        \mathsf u_i^{(k)}
        \gets
        \mathfrak s_{h,i}(\hat\mu_i^{(k-1)})
        \)
        \Comment{local 3D solve}

    \EndFor

    \Statex
    \Statex \(\triangleright\) \textbf{Global reconstruction and 1D update}

    \State Reconstruct
    \(\mathsf u^{(k)}\) from
    \(\{\mathsf u_i^{(k)}\}_{i\in\mathcal I}\)
    \Comment{global 3D field}

    \State \(\displaystyle
    \mathsf u_\Lambda^{(k)}
    \gets
    \mathfrak s_{\Lambda}(\mathsf{u}^{(k)}, \mu)
    \)
    \Comment{global 1D solve}

\EndFor

\Statex

\State \Return \((\mathsf u^{(k)},\mathsf u_\Lambda^{(k)})\)

\end{algorithmic}
\end{algorithm}
%\clearpage

%%%%%% ROM SECTION: work in progress

%titolo della sezione ?%
%Model order reduction strategy of the domain decomposition method%

\section{Neural model order reduction integrated with domain decomposition}
\label{sec:neural-dd-rom}

The localized learning framework of Section~\ref{sec:local_representability} relies on the assumption that, after identification of the subdomains with the common reference configuration $\Omega_\star$, the local solution maps belong to a common parametrized family. The domain decomposition formulation of Section~\ref{sec:fem_discretization} provides a concrete realization of this setting for the mixed-dimensional problem considered here. We now exploit this local representability property to construct neural surrogates acting on the common reference discretization.

The surrogate architecture is determined by two structural features of the localized DD operators. First, their inputs and outputs are spatially distributed finite element quantities defined on the reference mesh. Second, the localized physical parameter contains fine-scale geometric information associated with the embedded one-dimensional structure. These observations motivate the use of mesh-informed neural architectures and, for the local solution map, a reduced representation specifically designed to retain microstructure-induced solution components.

In what follows, we will introduce three surrogate models to replace key steps of the MS-DD-FOM algorithm. First, the local solution
surrogate $S_{\mathrm{rom}}$ plays a central role, as it directly
approximates the reference local solution map $\mathfrak s_\star$ introduced in
Section~\ref{sec:local_representability}, which is the main object of the localized learning framework.
However, replacing the local solver alone would still require fine-scale
operator evaluations to compute the Robin transmission data and the local
contributions to the coarse residual. For this reason, $S_{\mathrm{rom}}$
is complemented by transmission and coarse-residual surrogates
$T_{\mathrm{rom}}$ and $C_{\mathrm{rom}}$, respectively. 

Denoting by \(\widehat\mu_i^{(k)}\in\widehat{\mathcal P}\) the localized parameter associated with the \(i\)-th DD subproblem at iteration \(k\), that collects the local geometric information together with the external data induced by the DD iteration, we define below the three surrogate models that characterize the MS-DD-ROM algorithm.

\begin{description}
\item[\(S_{\rm rom}\):] the local solution surrogate is denoted by \(S_{\rm rom}\), and approximates the solution of the local DD subproblem, provided the localized DD data and the boundary and interface data, through the operator:
\[
    S_{\rm rom}:
    \widehat{\mathcal P}
    \longrightarrow
    \mathcal V_{h,i},
    \qquad
    S_{\rm rom}(\widehat\mu_i^{(k-1)})
    \approx
    \mathsf u_i^{(k)}.
\]
\item[\(T_{\rm rom}\):] we define the transmission surrogate as a non-intrusive operator approximating the state-to-transmission map:
\[
    T_{\rm rom}:
    \widehat{\mathcal P}\times \mathcal V_{h,i}
    \longrightarrow
    \mathcal V_{h,i}' ,
    \qquad
    T_{\rm rom}
    \left(
        \widehat\mu_i^{(k)},
        \mathsf u_i^{(k)}
    \right)
    \approx
    \mathsf f_i^{(k)}.
    \label{eq:Trom-definition}
\]

\item[\(C_{\rm rom}\):] the coarse surrogate is associated with the two-level correction. Let \(\mathsf r_c^{(k)}\in\mathcal V_c'\) denote the coarse residual entering the Galerkin problem and let \(\mathsf r_c^{(k)}=\sum_{i\in\mathcal I}\mathsf r_{c,i}^{(k)}\) be its decomposition into local contributions. The surrogate
\(C_{\rm rom}\) replaces the fine-scale operations required to evaluate and
project each local residual onto the coarse space. More precisely,
\[
C_{\rm rom}:\widehat{\mathcal P}\times\mathcal V_{h,i}\longrightarrow\mathcal V_c',
\qquad
C_{\rm rom}\!\left(\widehat\mu_i^{(k)},\mathsf u_i^{(k)}\right)\approx \mathsf r_{c,i}^{(k)},
\qquad
\sum_{i\in\mathcal I}C_{\rm rom}\!\left(\widehat\mu_i^{(k)},\mathsf u_i^{(k)}\right)\approx \mathsf r_c^{(k)}.
\label{eq:Crom-definition}
\]
The deterministic coarse problem on \(\mathcal V_c\) is retained in the online algorithm.
\end{description}

The combined use 
of the three surrogate operators yields a fully non-intrusive online
MS-DD-ROM: high-fidelity local operators are queried only offline to generate
training data, whereas the online DD iteration is advanced entirely through
neural evaluations and the deterministic solution of the low-dimensional
coarse Galerkin problem on $\mathcal V_c$. Thus, the proposed MS-DD-ROM is not
merely a surrogate of the local subproblem solver, but replaces all
fine-scale operator evaluations entering the online DD iteration.

\subsection{Neural architectures of the localized DD surrogates operators}
\label{subsec:ann-minn-podminn}

Since the localized DD quantities are represented as finite element fields on the common reference mesh, we employ mesh-informed neural networks (MINNs) \cite{franco2023mesh}, whose connectivity is constrained by the spatial organization of the underlying degrees of freedom. For the local solution map, we combine this local representation with a POD-based reduced space through the POD-MINN+ construction \cite{vitullo2024nonlinear}. The relevant architectural choices are specified below for each surrogate operator.

\subsubsection{Parametrization of the discrete local problems}
Here, we make explicit the algebraic representation of the parameter encoding the localized DD data  \(\widehat\mu_i^{(k)}\in\widehat{\mathcal P}\). After fixing a finite element reference space in \(\Omega_\star\), the components of \(\widehat\mu_i^{(k)}\) used by neural surrogates are represented by coefficient vectors. 

In particular, local geometric information is obtained from the embedded one-dimensional structure \(\Lambda_i\). From a computational standpoint, \(\Lambda_i\) is represented as a metric graph whose vertices are points in the three-dimensional physical space. Since this graph representation is not directly suitable as an input for finite-element-based neural architectures, we transform it into a spatial field defined on the local three-dimensional mesh. To this purpose we use the function mapping each point in the 3D domain $\Omega_\star$ to its distance from the closest point of the 1D graph. We represent this function as a linear and piecewise continuous finite element function in $\mathcal{V}_{h,i}$. The corresponding vector of degrees of freedom is denoted by $\mathsf d_i\in\mathbb R^{N_h}$.

The local embedded-dimensional state is transferred to the three-dimensional finite element mesh through the interpolation operator
\[
    \mathsf u_{\Lambda\to h,i}^{(k)}
    =
    \mathsf E^\Lambda_{h,i}\mathsf u_{\Lambda,i}^{(k)}
    \in \mathbb R^{N_h},
\]
where $ \mathsf E^\Lambda_{h,i} \in \mathbb R^{N_h\times N_{\Lambda,i}}$ maps the local embedded-dimensional degrees of freedom onto the finite element mesh of the three-dimensional subdomain.
Moreover, we define the aggregated contributions of Robin transmission data as
\[
    \mathsf f_{\mathrm{agg},i}^{(k)}
    =
    \sum_{j\in\mathcal N(i)}
    \mathsf L_i \mathsf L_j^\top
    \mathsf f_j^{(k)}
    \in \mathbb R^{N_h}.
\]

\subsubsection{Local solution surrogate: a microstructure-informed reduced representation}
\label{sec:local_solution_surrogate}

Among the localized operators introduced above, the local solution map poses the main model-reduction challenge. Indeed, for a localized parameter
\(\widehat{\boldsymbol\mu}_i^{(k)} = (\mathsf d_i, \mathsf u_{\Lambda\to h,i}^{(k)}, \mathsf f_{\mathrm{agg},i}^{(k)})\), 
the corresponding Robin subproblem defines a map with values in the high-dimensional finite element space. Using the notation introduced in Section~\ref{sec:local_representability}, we denote the discrete reference parameter-to-solution map by
\[
    \mathfrak{s}_\star:
    \widehat{\mathcal P}
    \longrightarrow
    \mathbb R^{N_h^\star},
    \qquad
    \mathfrak{s}_\star
    \left(
        \widehat{\boldsymbol\mu}_i^{(k)}
    \right)
    =
    \mathsf u_i^{(k+1)},
\]
with \(N_h^\star\) the number of degrees of freedom of the local finite element space. The local solution surrogate \(S_{\rm rom}\) is designed to approximate this reference solution map. We start from standard projection-based reduction seeking a fixed low-dimensional linear space capable of approximating the complete family of local solutions. Let
\(\mathsf V = [ \boldsymbol\varphi_1,\ldots,\boldsymbol\varphi_{N_{rb}}] \in \mathbb R^{N_h^\star\times N_{rb}}\),
denote the proper orthogonal decomposition (POD) basis obtained from local high-fidelity solution snapshots.
The orthogonal projection of a local solution onto \(\mathcal{V}_{rb}:={\rm range}(\mathsf V)\) is \(\Pi_{rb}\mathsf u_i = \mathsf V\mathsf V^\top\mathsf u_i\),
with reduced coordinates
\(\mathsf a_i := \mathsf V^\top\mathsf u_i \in \mathbb R^{N_{rb}}\).
A non-intrusive reduced model such as the POD-NN method presented in \cite{hesthaven2018non} then approximates the parameter-to-coefficient map by a neural operator
\[
    \alpha_{\theta_S^{\rm rb}}:
    \widehat{\mathcal P}
    \longrightarrow
    \mathbb R^{N_{rb}},
\]
yielding the POD reconstruction
\(\mathsf u_{{\rm rb},i} (\widehat{\boldsymbol\mu}_i)
= \mathsf V \alpha_{\theta_S^{\rm rb}}
(\widehat{\boldsymbol\mu}_i)\).
Since the POD basis is Euclidean-orthonormal, the approximation error admits the orthogonal decomposition
\begin{equation*}
    \|
        \mathsf u_i
        -
        \mathsf V
        \alpha_{\theta_S^{\rm rb}}
        (
            \widehat{\boldsymbol\mu}_i
        )
    \|_2^2
    =
    \|
        \mathsf u_i
        -
        \mathsf V\mathsf V^\top\mathsf u_i
    \|_2^2
    +
    \|
        \mathsf V^\top\mathsf u_i
        -
        \alpha_{\theta_S^{\rm rb}}
        (
            \widehat{\boldsymbol\mu}_i
        )
    \|_2^2 .
\end{equation*}
Thus, increasing the expressive power of the reduced-coordinate regression can only reduce the second contribution. The first term is the irreducible projection error associated with the fixed space \(V_{rb}\).

This distinction is particularly important for the multiscale problems considered here. The embedded structure \(\Lambda_i\) induces localized solution features whose position and spatial organization vary with the local geometry. A fixed linear reduced space must represent all such configurations simultaneously. Consequently, microstructure-induced components may require a comparatively large number of POD modes even when, for each fixed geometry, the corresponding local feature is spatially localized and highly structured. In this regime, the limitation is not only the approximation of the reduced coefficients, but the use of a single geometry-independent approximation space.

To address this issue, we enrich the POD reconstruction by a correction explicitly driven by the local microstructure descriptor,
\[
    c_{\theta_S^{\rm c}}:
    \mathbb R^{N_h}
    \longrightarrow
    \mathbb R^{N_h},
    \qquad
    \mathsf d_i
    \longmapsto
    c_{\theta_S^{\rm c}}(\mathsf d_i).
\]
Denoting the trainable parameters as \(\theta_S= (\theta_S^{\rm rb},\theta_S^{\rm c})\), the resulting local solution surrogate is
\begin{equation}\label{eq:Srom-podminnplus}
    S_{\rm rom}^{\theta_S}
    \left(
        \mathsf d_i,
        \mathsf u_{\Lambda\to h,i}^{(k)},
        \mathsf f_{\mathrm{agg},i}^{(k)}
    \right)
    =
    \mathsf V\,
    \alpha_{\theta_S^{\rm rb}}
    \left(
        \mathsf d_i,
        \mathsf u_{\Lambda\to h,i}^{(k)},
        \mathsf f_{\mathrm{agg},i}^{(k)}
    \right)
    +
    c_{\theta_S^{\rm c}}(\mathsf d_i).
\end{equation}

The representation in \eqref{eq:Srom-podminnplus} admits a useful multiscale interpretation. The fixed POD space \(\mathcal{V}_{rb}\) captures solution components that are recurrent across the localized parameter family and can therefore be efficiently represented through a common linear basis. The closure term instead introduces a fine-scale correction whose spatial structure depends explicitly on the embedded geometry. Accordingly, for each admissible descriptor \(\mathsf d\), we may associate the enriched affine space
\[
    \mathcal V^+_{rb}(\mathsf d)
    :=
    \mathcal V_{rb}
    +
    c_{\theta_S^{\rm c}}(\mathsf d)
    =
    \left\{
        \mathsf V\mathsf a
        +
        c_{\theta_S^{\rm c}}(\mathsf d)
        :
        \mathsf a\in\mathbb R^{N_{rb}}
    \right\}.
\]
As the microstructure varies, the correction displaces the fixed POD space and generates the nonlinear approximation manifold
\[
    \mathcal C_{\theta_S^{\rm c}}
    :=
    \bigcup_{\mathsf d\in\mathcal D}
    \mathcal V_{rb}(\mathsf d)
    =
    \left\{
        \mathsf V\mathsf a
        +
        c_{\theta_S^{\rm c}}(\mathsf d)
        :
        \mathsf a\in\mathbb R^{N_{rb}},
        \ \mathsf d\in\mathcal D
    \right\}
    \subset
    \mathbb R^{N_h},
\]
where \(\mathcal D\) denotes the set of admissible local geometric descriptors. The nonlinear manifold \(\mathcal C_{\theta_S^{\rm c}}\) therefore provides a microstructure-dependent enrichment of the fixed POD space.

This viewpoint creates a conceptual connection with multiscale numerical methods. In many classical multiscale constructions, unresolved local heterogeneities modify or enrich a coarse approximation through basis functions or correctors that depend on the fine-scale medium. Here, the role of the microstructure-dependent correction is learned non-intrusively from local high-fidelity data: the map
\(\mathsf d \longmapsto c_{\theta_S^{\rm c}}(\mathsf d)\)
associates the local geometric configuration with a fine-scale solution component not efficiently represented by the common POD space. We do not claim an equivalence with classical multiscale basis or corrector constructions; rather, the analogy provides an interpretation of the learned closure as a data-driven, microstructure-informed enrichment of a coarse reduced representation.

The two components of \(S_{\rm rom}\) are implemented using mesh-informed neural networks. A mesh-informed layer can be interpreted as a trainable local operator between finite element coefficient spaces, with a strong inductive bias based on the connectivity of the underlying computational mesh. The interested reader is referred to \cite{franco2023mesh} for additional details. The reduced coefficients are predicted through a three-branch POD-MINN architecture. Each branch acts on one of the finite element fields parametrizing the local Robin problem. More precisely, we introduce three mesh-informed encoders
\(\phi_1,\phi_2,\phi_3:
    \mathbb R^{N_h}
    \longrightarrow
    \mathbb R^{N_{rb}}\),
with architecture
\[
    \phi_\ell:
    \mathbb R^{N_h}
    \xrightarrow{\;{\rm MINN}(r_1^S)\;}
    \mathbb R^{N_{\rm aux}}
    \xrightarrow{\;{\rm MINN}(r_2^S)\;}
    \mathbb R^{N_{\rm aux}}
    \xrightarrow{\;{\rm Dense}\;}
    \mathbb R^{N_{rb}},
    \qquad
    \ell=1,2,3.
\]
Here \(N_{\rm aux}\) denotes the number of degrees of freedom of an auxiliary coarse mesh used internally by the mesh-informed layers and \(r_1^S,r_2^S\) denote the interaction radii defining the sparsity pattern of the mesh-informed layers. In the numerical implementation, both local encoder layers use radius \(r_1^S=r_2^S=0.3\). The three encoders are applied as
\(\phi_1(\mathsf d_i), \ \phi_2(\mathsf u_{\Lambda\to h,i}^{(k)}), \ \phi_3(\mathsf f_{\mathrm{agg},i}^{(k)})\),
and their outputs are concatenated and passed to a dense decoder,
\begin{equation*}
    \alpha_{\theta_S^{\rm rb}}
    \left(
        \mathsf d_i,
        \mathsf u_{\Lambda\to h,i}^{(k)},
        \mathsf f_{\mathrm{agg},i}^{(k)}
    \right)
    :=
    {\rm decoder}
    \big(
        \phi_1(\mathsf d_i),
        \phi_2(\mathsf u_{\Lambda\to h,i}^{(k)}),
        \phi_3(\mathsf f_{\mathrm{agg},i}^{(k)})
    \big),
\end{equation*}
where
\({\rm decoder}:
    \mathbb R^{3N_{rb}}
    \longrightarrow
    \mathbb R^{N_{rb}}\)
has the architecture
\[
    \mathbb R^{3N_{rb}}
    \xrightarrow{\;{\rm Dense}\;}
    \mathbb R^{2N_{rb}}
    \xrightarrow{\;{\rm Dense}\;}
    \mathbb R^{2N_{rb}}
    \xrightarrow{\;{\rm Dense}\;}
    \mathbb R^{N_{rb}}.
\]
In the implementation considered here, the local encoder layers use hyperbolic tangent activations and are followed by linear maps onto the POD-coordinate space. The decoder uses \(0.1\)-leaky ReLU activations in the hidden layers and a linear output layer. The encoder weights are initialized according to the normalized initialization of \cite{GlorotBengio}, whereas the decoder weights are initialized at zero.

The microstructure-informed closure is implemented as a local mesh-informed correction,
\[
    c_{\theta_S^{\rm c}}:
    \mathbb R^{N_h}
    \xrightarrow{\;{\rm MINN(r_c^S)}}
    \mathbb R^{N_h},
\]
where the locality constraint is chosen such that changes in the descriptor of the embedded structure generate spatially localized corrections on the reference mesh. The radius \(r_c^S\) controls the spatial support of the geometry-driven correction. In the numerical implementation, we use \(r_c^S=0.2\).
Finally, the closure is initialized at zero. Hence, the initial approximation manifold coincides with the fixed POD space,
\(\mathcal C_{\theta_S^{\rm c}} = V_{rb} \ \text{when} \ c_{\theta_S^{\rm c}}=0\).
The reduced POD-MINN component is trained first and subsequently frozen. The closure is then trained to recover the solution component left unexplained by the learned POD reconstruction. In this sense, the training procedure progressively enriches the fixed reduced space with microstructure-dependent fine-scale information.

\subsubsection{Transmission surrogate architecture}
This surrogate model approximates the Robin transmission vector associated with the local state. \(T_{\rm rom}^{\theta_T}\) is implemented as the sum of a learned residual map and an explicit Robin mass term:
\[
    T_{\rm rom}^{\theta_T}
    \left(
        \mathsf u_i^{(k)}
    \right)
    :=
    T_{\rm res}^{\theta_T}
    \left(
        \mathsf u_i^{(k)}
    \right)
    +
    \rho \mathsf G_i^\Gamma
    \mathsf u_i^{\Gamma_i,(k)},
    \qquad 
    T_{\rm rom}^{\theta_T}
    \left(
        \mathsf u_i^{(k)}
    \right)
    \approx
    \mathsf f_i^{(k)}.
\]
where $\theta_T$ refers to the set of trainable parameters of the neural architectures. The learned component satisfies
\[
    T_{\rm res}^{\theta_T}:
    \mathbb R^{N_h}
    \longrightarrow
    \mathbb R^{N_h},
\]
and is implemented as a mesh-informed neural network, with architecture
\[
    \mathbb R^{N_h}
    \xrightarrow{\;{\rm MINN(r_1^T)}\;}
    \mathbb R^{N_{\rm aux}}
    \xrightarrow{\;{\rm MINN(r_2^T)}\;}
    \mathbb R^{N_{\rm aux}}
    \xrightarrow{\;{\rm MINN(r_3^T)}\;}
    \mathbb R^{N_h},
\]
where \(N_h=\dim\mathcal V_{h,i}\), \(N_{\rm aux}\) is the number of degrees of freedom of an auxiliary mesh and  \(r_1^T,r_2^T,r_3^T\) denote the interaction radii defining the sparsity pattern of the transmission surrogate layers. In the numerical implementation, we consider \(r_1^T=0.3\), \(r_2^T=0.2\), and \(r_3^T=0.15\). The local layers use hyperbolic tangent activations, while the last layer is linear so that the network can represent signed residual values \cite{Kutnyok}. The trainable parameters are collected in \(\theta_T\), and are initialized following the normalized initialization proposed in \cite{GlorotBengio}.
Although the operator-level map depends on the localized parameter \(\widehat\mu_i^{(k)}\), in the present algebraic implementation the network is queried as a state-to-residual map. The dependence on the local microstructure and on the embedded-dimensional forcing is therefore encoded through the MS-DD-FOM trajectories used to generate the training distribution.

\subsubsection{Coarse residual surrogate architecture}
The third surrogate approximates the algebraic coarse residual contribution entering the two-level correction. We recall that the exact local contribution is given by
\[
    \mathsf r_{c,i}^{(k)}
    =
    \mathsf P_i^\top
    \mathsf D_i
    (\mathsf b_i^{(k)}
    -
    \mathsf A_i\mathsf u_i^{(k)}).
\]
Thus, as previously defined, the coarse surrogate is trained to approximate the composed map from the local state to the local coarse residual contribution and the global coarse right-hand side is then assembled by summing the predicted local contributions,
\[
    C_{\rm rom}^{\theta_C}:
    \mathbb R^{N_h}
    \longrightarrow
    \mathbb R^{N_c},
    \qquad
    C_{\rm rom}^{\theta_C}
    \left(
        \mathsf u_i^{(k)}
    \right)
    \approx
    \mathsf r_{c,i}^{(k)},
    \qquad
    \sum_{i\in\mathcal I}
    C_{\rm rom}^{\theta_C}
    \left(
        \mathsf u_i^{(k)}
    \right)
    \approx
    \mathsf r_c^{(k)},
\]
where \(N_c=\dim\mathcal V_c\) is the dimension of the coarse residual coefficient vector. 

Denoting by $\theta_C$ the set of trainable parameters, the map \(C_{\rm rom}^{\theta_C}\) is implemented as a MINN encoder followed by dense layers. In abstract form,
\[
    \mathbb R^{N_h}
    \xrightarrow{\;{\rm MINN(r_1^C)}\;}
    \mathbb R^{N_{\rm aux}}
    \xrightarrow{\;{\rm MINN(r_2^C)}\;}
    \mathbb R^{N_{\rm aux}}
    \xrightarrow{\;{\rm Dense}\;}
    \mathbb R^{m_1}
    \xrightarrow{\;{\rm Dense}\;}
    \mathbb R^{m_2}
    \xrightarrow{\;{\rm Dense}\;}
    \mathbb R^{N_c},
\]
where \(r_1^C\) and \(r_2^C\) define the supports of the local mesh-informed encoder layers. In the numerical implementation, these are set to \(r_1^C=r_2^C=0.3\). As for the transmission map, the coarse surrogate is queried as a state-to-residual map. The dependence on the local geometry and on the embedded-dimensional forcing is therefore inherited through the MS-DD-ROM
training trajectories.

The input-output structure of the three surrogates is summarized in Table~\ref{tab:surrogate-tasks}.

\begin{table}[htb!]
\centering
\begin{tabular}{lll}
\toprule
Surrogate & Inputs & Outputs \\
\midrule
\(T_{\rm rom}\) &
\(\mathsf u_{i}^{(k)}\) &
local Robin transmission datum \(\mathsf f_{i}^{(k)}\) \\
\(S_{\rm rom}\) &
\((\mathsf d_i,\mathsf u_{\Lambda\to h,i}^{(k)},\mathsf f_{{\rm agg},i}^{(k)})\) &
next local state \(\mathsf u_{i}^{(k+1)}\) \\
\(C_{\rm rom}\) &
\(\mathsf u_{i}^{(k)}\) &
local coarse residual \( \mathsf r_{c,i}^{(k)}\) \\
\bottomrule
\end{tabular}
\caption{Input-output structure of the three surrogate tasks.}
\label{tab:surrogate-tasks}
\end{table}

The definition of the neural architectures alone is not sufficient to guarantee a reliable MS-DD-ROM iteration. The surrogate maps are data-driven approximations and should not be regarded as distribution-invariant operators: their accuracy is tied to the distribution of the local states, residuals, and interface data used during training. Since these quantities are themselves generated by an iterative DD algorithm, replacing one high-fidelity component by a surrogate changes the input distribution seen by the remaining components. The training procedure must therefore be designed at the algorithmic level, so that each surrogate is trained on data representative of the MS-DD-ROM regime in which it will be queried online, to reduce the offline-online distribution shift between the offline training data and the online MS-DD-ROM trajectories. This motivates the cascaded workflow described next.

\subsection{Cascaded training strategy for the neural surrogates} 
\label{sec:cascade_training}
The three neural surrogate operators are not trained independently on arbitrary local samples. They are trained in cascade, following the order in which the corresponding quantities are generated and used within the two-level DD iteration. This choice is motivated by the perturbation viewpoint introduced in Section~\ref{sec:local_perturbation_theory}. Once a surrogate is inserted into the DD loop, the subsequent states, interface data, and residual quantities are no longer those produced by the exact MS-DD-FOM, but by a perturbed iteration. Training the next surrogate on exact MS-DD-FOM data only would therefore control its error on the wrong input regime and would increase the effective surrogate-induced perturbation entering the local DD update. Therefore, each training set must be generated from the algorithmic configuration that matches the online use of the corresponding surrogate.

%The organization of the cascade training procedure, compltely detailed in the present section and in Subsection~\ref{subsec:Training_losses}, is  summarized in Figure~\ref{fig:cascade}. 
The cascade training procedure is summarized in Figure~\ref{fig:cascade} and detailed below, including the definition of training losses (see Subsection~\ref{subsec:Training_losses}).
\begin{figure}
    \centering
    \includegraphics[width=0.9\linewidth]{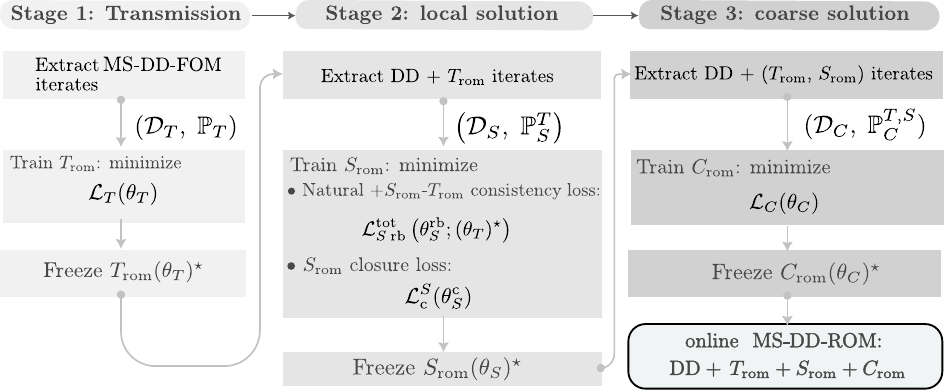}
    \caption{Cascaded training workflow. The transmission surrogate is trained on MS-DD-FOM trajectories; the local solution surrogate is trained with the transmission surrogate already deployed; the coarse-residual surrogate is trained on MS-DD-ROM trajectories generated with both transmission and solution surrogates active.}
    \label{fig:cascade}
\end{figure}
The ordering is not meant to be unique; it is an incremental training choice in which the high-fidelity components are replaced one at a time and each surrogate is trained under the input regime induced by the components already deployed. Alternative training orderings, or even joint fine-tuning strategies, are in principle compatible with the same convergence argument, but are not investigated in this work. First, \(T_{\rm rom}\) is trained on MS-DD-FOM data to approximate the Robin transmission datum \(\mathsf f_i^{(k)}\), by learning its Neumann/DtN residual component and retaining the Robin trace term explicitly as
\[
    \widetilde{\mathsf f}_i^{(k)}
    =
    T_{\rm rom}^{\theta_T}
    \left(
        \mathsf u_i^{(k)}
    \right).
\]
Operationally, the deployed transmission map outputs the Robin datum \(\mathsf f_i^{(k)}\). Since the Robin trace contribution is deterministic, the trainable component is equivalently trained on the residual contribution \(\mathsf r_i^{(k)}\). Taking into account data distributions induced by the corresponding DD trajectories, let \(\mathbb P_T\) denote the distribution of the transmission training pairs generated by the MS-DD-FOM, such that
\[
    \left(
        \mathsf u_i^{(k)};
        \mathsf f_i^{(k)}
    \right)
    \sim
    \mathbb P_T .
\]
Second, \(S_{\rm rom}^{\theta_S}\) is trained on data generated by a MS-DD-FOM in which \(T_{\rm rom}^{\theta_T}\) is already active, reconstructing the Robin transmission data as
\[
    \widetilde{\mathsf f}_{\mathrm{agg},i}^{(k)}
    =
    \sum_{j\in\mathcal N(i)}
    \mathsf L_i\mathsf L_j^\top \widetilde{\mathsf f}_j^{(k)}.
\]
Thus, the local solver surrogate \(S_{\rm rom}^{\theta_S}\) is exposed during training to aggregated Robin data already perturbed by the learned transmission map $T_{\rm rom}^{\theta_T}$. The corresponding training pairs are then induced by the resulting distribution \(\mathbb P_S^T\),
\[
    \left(
        \mathsf d_i,
        \mathsf u_{\Lambda\to h,i}^{(k)},
        \widetilde{\mathsf f}_{\mathrm{agg},i}^{(k)};
        \mathsf u_i^{(k+1)}
    \right)
    \sim
    \mathbb P_S^T.
\]

Finally, \(C_{\rm rom}\) is trained after both \(T_{\rm rom}\) and
\(S_{\rm rom}\) have been deployed, because the coarse residual is evaluated on MS-DD-ROM states produced by the reduced transmission and local solution steps. This entails
\[
    \mathsf r_{c,i}^{(k)}
    =
    \mathsf P_i^\top
    \mathsf D_i
    \widetilde{\mathsf r}_i^{(k)},
    \qquad
    \widetilde{\mathsf r}_i^{(k)}
    =
    \mathsf b_i^{(k)}
    -
    \mathsf A_i\widetilde{\mathsf u}_i^{(k)} .
\]
The MS-DD-ROM trajectory induces the coarse-residual distribution
\(\mathbb P_C^{T,S}\), so that the corresponding training pairs are
\[
    \left(
        \widetilde{\mathsf u}_{i}^{(k)};
        \mathsf P_i^\top \mathsf D_i
        \widetilde{\mathsf r}_{i}^{(k)}
    \right)
    \sim
    \mathbb P_C^{T,S}.
\]
The empirical losses approximate risks associated with the distributions actually visited online. Training \(S_{\rm rom}\) or \(C_{\rm rom}\) only on exact MS-DD-FOM trajectories would control the error on distributions that are not necessarily those encountered by the deployed MS-DD-ROM.
The resulting cascade
%—summarized in Table~\ref{tab:training-cascade}—
reduces the offline-online distribution shift between the offline training data and the online MS-DD-ROM trajectories. Equivalently, it aims at controlling the effective surrogate-induced perturbation entering the DD iteration, rather than only minimizing isolated regression errors on states produced by the fully assembled MS-DD-FOM.

\begin{comment}
\begin{table}[htb!]
\centering
\begin{tabular}{llll}
\toprule
Stage & Algorithm generating data & Induced law & Training pairs \\
\midrule
1 &
DD-FOM &
\(\mathbb P_T\) &
\(\mathsf u_i^{(k)}\mapsto \mathsf f_i^{(k)}\) \\
2 &
DD-FOM with \(T_{\rm rom}\) &
\(\mathbb P_S^T\) &
\((\mathsf d_i,\mathsf u_{\Lambda\to h,i}^{(k)},\widetilde{\mathsf f}_{\mathrm{agg},i}^{(k)})\mapsto \mathsf u_i^{(k+1)}\) \\
3 &
DD-ROM with \(T_{\rm rom},S_{\rm rom}\) &
\(\mathbb P_C^{T,S}\) &
\(\widetilde{\mathsf u}_i^{(k)}
\mapsto
\mathsf P_i^\top \mathsf D_i\widetilde{\mathsf r}_i^{(k)}\) \\
\bottomrule
\end{tabular}
\caption{Cascaded training workflow for the DD-ROM operators. Each surrogate is trained on data generated by the algorithmic configuration in which it is used online.}
\label{tab:training-cascade}
\end{table}
\end{comment}

\begin{comment}
\begin{figure}[h!]
    \centering
    \begingroup
    \def\svgwidth{\textwidth}
    \input{Figures/cascade_training.pdf_tex}
    \endgroup
    \caption{\color{red} v0, link alle equazionin da aggiungere; troppo piccolo, da riorganizzare\color{black}}
    \label{fig:cascade-training}
\end{figure}

\begin{figure}[h!]
    \centering
    \begingroup
    \def\svgwidth{\textwidth}
    \input{Figures/cascade_training_v1.pdf_tex}
    \endgroup
    \caption{\color{red} versione 1, link alle equazionin da aggiustare...forse invece di riportare l'intera espressione della loss, posso fare riferimento all'argim problem o aequazioni affini\color{black}}
    \label{fig:cascade-training}
\end{figure}
\end{comment}
\subsubsection{Training losses of the neural surrogates}
\label{subsec:Training_losses}
We now describe the loss functions used to train the neural surrogates. In the following, errors between finite element fields defined on the local subdomain are measured in a discrete \(L^2(\Omega_\star)\)-norm,
\[
    \|\mathsf v\|_{L^2(\Omega_\star),h}^2
    :=
    \mathsf v^\top \mathsf G \mathsf v,
\]
where \(\mathsf G\) is the local mass matrix. This convention is used for local solution fields, residual contributions, and Robin transmission vectors represented on the local finite element mesh. Euclidean norms are retained for reduced POD coefficients.
The datasets are collected over local subdomains, DD iterations, and
geometric realizations. We use \(i\in\mathcal I\) to index the local subdomain, \(k\in\mathcal K_g\) to index the DD iteration along the trajectory associated with a given geometry, and \(g\in\mathcal G\) to index the sampled embedded graph configuration. Accordingly, each sample is associated with a local graph \(\Lambda_i^{(g)}\), its descriptor \(\mathsf d_i^{(g)}\), and the corresponding DD trajectory generated for that geometry.

\paragraph{Task 1: learning the transmission surrogate}
\label{subsec:task-Trom}

Although the deployed transmission surrogate outputs the Robin datum \(\mathsf f_i^{(k)}\), only its residual component is learned. The deterministic Robin trace contribution is added after neural inference, as described in the architecture of \(T_{\rm rom}\). The corresponding training set is
\[
    \mathcal D_T
    =
    \left\{
    \left(
        \mathsf u_i^{(k)};
        \mathsf f_i^{(k)}
    \right)
    \right\}_{i\in\mathcal I,\; k\in\mathcal K_g,\; g\in\mathcal G}.
\]
The empirical loss is
\begin{equation}
    \mathcal L_T(\theta_T)
    =
    \frac{1}{|\mathcal D_T|}
    \sum_{\mathcal D_T}
    \left\|
    T_{\rm rom}^{\theta_T}
    \left(
        \mathsf u_i^{(k)}
    \right)
    -
    \mathsf f_i^{(k)}
    \right\|_{L^2(\Omega_\star),h}^2,
    \label{eq:loss-Trom}
\end{equation}
recalling that
\[
    T_{\rm rom}^{(\theta_T)}
    \left(
        \mathsf u_i^{(k)}
    \right)
    =
    T_{\rm res}^{(\theta_T)}
    \left(
        \mathsf u_i^{(k)}
    \right)
    +
    \rho \mathsf G_i^\Gamma
    \mathsf u_i^{\Gamma_i,(k)}.
    \label{eq:Trom-deployed}
\]
The trained parameters are defined by
\[
    (\theta_T)^\star
    \in
    \arg\min_{\theta_T}
    \mathcal L_T(\theta_T).
    \label{eq:Trom-optimal-parameters}
\]

\paragraph{Task 2: learning the local solution operator}
\label{subsec:task-Srom}

The local solution surrogate is trained on data generated by the cascaded DD
configuration in which the transmission surrogate is already active:
\[
    \mathcal D_S
    =
    \left\{
    \left(
        \mathsf d_i,
        \mathsf u_{\Lambda\to h,i}^{(k)},
        \widetilde{\mathsf f}_{\mathrm{agg},i}^{(k)};
        \mathsf u_i^{(k+1)}
    \right)
    \right\}_{i\in\mathcal I,\; k\in\mathcal K_g,\; g\in\mathcal G}.
\]
The training of \(S_{\rm rom}\) is performed in two steps. First, the POD-MINN component is trained while the closure branch is inactive, namely
\(\mathsf c_{\theta_S^{\rm c}}=0\). Here
\(
    \mathsf a_i^{(k+1)}
    :=
    \mathsf V^\top \mathsf u_i^{(k+1)}
    \in\mathbb R^{N_{rb}}
\)
denotes the POD coefficient vector of the high-fidelity local solution. The data-fidelity term is imposed on the
reduced coordinates:
\[
    \mathcal L_{\rm rb}^S(\theta_S^{\rm rb})
    =
    \frac{1}{|\mathcal D_S|}
    \sum_{\mathcal D_S}
    \left\|
    \alpha_{\theta_S^{\rm rb}}
    \left(
        \mathsf d_i,
        \mathsf u_{\Lambda\to h,i}^{(k)},
        \widetilde{\mathsf f}_{\mathrm{agg},i}^{(k)}
    \right)
    -
    \mathsf a_i^{(k+1)}
    \right\|_2^2 .
    \label{eq:loss-rb-Srom}
\]
This reduced data-fidelity term is complemented by two local consistency losses involving the pretrained transmission surrogate \(T_{\rm rom}^{(\theta_T)^\star}\).
During this stage, \((\theta_T)^\star\) is kept fixed and only the parameters \(\theta_S^{\rm rb}\) of the POD-MINN component are updated.

Let
\[
    \widehat{\mathsf u}_{i,{\rm rb}}^{(k+1)}
    :=
    \mathsf V
    \alpha_{\theta_S^{\rm rb}}
    \left(
        \mathsf d_i,
        \mathsf u_{\Lambda\to h,i}^{(k)},
        \widetilde{\mathsf f}_{\mathrm{agg},i}^{(k)}
    \right)
\]
be the reconstructed POD-MINN prediction. The associated Robin transmission
datum is evaluated through the frozen transmission surrogate:
\[
    \widehat{\mathsf f}_{i,{\rm rb}}^{(k+1)}
    =
    T_{\rm rom}^{(\theta_T)^\star}
    \left(
        \widehat{\mathsf u}_{i,{\rm rb}}^{(k+1)}
    \right).
\]
The first consistency term enforces compatibility in the direction
\[
    \widetilde{\mathsf f}_{\mathrm{agg},i}^{(k)}
    \xrightarrow{\;S_{\rm rom}^{\theta_S^{\rm rb}}\;}
    \widehat{\mathsf u}_{i,{\rm rb}}^{(k+1)}
    \xrightarrow{\;T_{\rm rom}^{(\theta_T)^\star}\;}
    \widehat{\mathsf f}_{i,{\rm rb}}^{(k+1)} .
\]
It is defined as
\[
    \mathcal L_{TS}(\theta_S^{\rm rb};(\theta_T)^\star)
    =
    \frac{1}{|\mathcal D_S|}
    \sum_{\mathcal D_S}
    \left\|
    \widehat{\mathsf f}_{i,{\rm rb}}^{(k+1)}
    -
    \mathsf f_i^{(k+1)}
    \right\|_{L^2(\Omega_\star),h}^2.
    \label{eq:loss-TS}
\]
The second consistency term acts in the reverse direction. Starting from the
high-fidelity local state \(\mathsf u_i^{(k)}\), we evaluate the frozen
transmission surrogate
\[
    \overline{\mathsf f}_i^{(k)}
    =
    T_{\rm rom}^{(\theta_T)^\star}
    \left(
        \mathsf u_i^{(k)}
    \right).
\]
After aggregation,
\[
    \overline{\mathsf f}_{\mathrm{agg},i}^{(k)}
    =
    \sum_{j\in\mathcal N(i)}
    \mathsf L_i\mathsf L_j^\top
    \overline{\mathsf f}_j^{(k)} ,
\]
the reverse consistency loss is defined as
\[
    \mathcal L_{ST}(\theta_S^{\rm rb};(\theta_T)^\star)
    =
    \frac{1}{|\mathcal D_S|}
    \sum_{\mathcal D_S}
    \left\|
    \mathsf V
    \alpha_{\theta_S^{\rm rb}}
    \left(
        \mathsf d_i,
        \mathsf u_{\Lambda\to h,i}^{(k)},
        \overline{\mathsf f}_{\mathrm{agg},i}^{(k)}
    \right)
    -
    \mathsf u_i^{(k+1)}
    \right\|_{L^2(\Omega_\star),h}^2 .
    \label{eq:loss-ST}
\]
The first training step is therefore defined by
\[
    \left(\theta_S^{\rm rb}\right)^\star
    \in
    \arg\min_{\theta_S^{\rm rb}}
    \mathcal L_{S,{\rm rb}}^{\rm tot}
    \left(\theta_S^{\rm rb};(\theta_T)^\star\right),
\]
with
\begin{equation}
    \mathcal L_{S,{\rm rb}}^{\rm tot}
    \left(\theta_S^{\rm rb};(\theta_T)^\star\right)
    =
    \omega_{\rm rb}\mathcal L_{\rm rb}^S(\theta_S^{\rm rb})
    +
    \omega_{TS}\mathcal L_{TS}
    \left(\theta_S^{\rm rb};(\theta_T)^\star\right)
    +
    \omega_{ST}\mathcal L_{ST}
    \left(\theta_S^{\rm rb};(\theta_T)^\star\right).
    \label{eq:loss-Srom-rb-total}
\end{equation}
The two consistency contributions act as local compatibility constraints between the learned reduced solution operator and the pretrained transmission surrogate, rather than as a joint training of the two maps.
Once the POD-MINN component has been trained, its parameters are frozen and a
second training step is performed for the closure branch. The closure is trained
to recover the component of the high-fidelity local solution not represented by
the learned POD reconstruction, using the local geometric descriptor as input.
For each training sample, we define the closure target as
\[
    \mathsf c_i^{(k+1)}
    :=
    \mathsf u_i^{(k+1)}
    -
    \mathsf V
    \alpha_{(\theta_S^{\rm rb})^\star}
    \left(
        \mathsf d_i,
        \mathsf u_{\Lambda\to h,i}^{(k)},
        \widetilde{\mathsf f}_{\mathrm{agg},i}^{(k)}
    \right).
\]
The closure loss is
\begin{equation}
    \mathcal L_{\rm c}^S(\theta_S^{\rm c})
    =
    \frac{1}{|\mathcal D_S|}
    \sum_{\mathcal D_S}
    \left\|
    c_{\theta_S^{\rm c}}
    \left(
        \mathsf d_i
    \right)
    -
    \mathsf c_i^{(k+1)}
    \right\|_{L^2(\Omega_\star),h}^2 .
    \label{eq:loss-closure-Srom}
\end{equation}
The closure parameters are defined by
\[
    \left(\theta_S^{\rm c}\right)^\star
    \in
    \arg\min_{\theta_S^{\rm c}}
    \mathcal L_{\rm c}^S(\theta_S^{\rm c}),
\]
with \(\left(\theta_S^{\rm rb}\right)^\star\) kept fixed.
After the two training steps, the deployed local solver surrogate is
\[
    S_{\rm rom}^{\theta_S^\star}
    \left(
        \mathsf d_i,
        \mathsf u_{\Lambda\to h,i}^{(k)},
        \widetilde{\mathsf f}_{\mathrm{agg},i}^{(k)}
    \right)
    =
    \mathsf V
    \alpha_{(\theta_S^{\rm rb})^\star}
    \left(
        \mathsf d_i,
        \mathsf u_{\Lambda\to h,i}^{(k)},
        \widetilde{\mathsf f}_{\mathrm{agg},i}^{(k)}
    \right)
    +
    c_{(\theta_S^{\rm c})^\star}
    \left(
        \mathsf d_i
    \right),
\]
where
\[
    \theta_S^\star
    :=
    \left(
        (\theta_S^{\rm rb})^\star,
        (\theta_S^{\rm c})^\star
    \right).
\]
\paragraph{Task 3: learning the coarse residual operator}
\label{subsec:task-Crom}

The third surrogate approximates the local contribution to the algebraic coarse residual entering the two-level correction. The training set is
\[
    \mathcal D_C
    =
    \left\{
    \left(
        \widetilde{\mathsf u}_i^{(k)};
        \mathsf r_{c,i}^{(k)}
    \right)
    \right\}_{i\in\mathcal I,\; k\in\mathcal K_g,\; g\in\mathcal G},
\]
with
\[
    \widetilde{\mathsf r}_i^{(k)}
    =
    \mathsf b_i^{(k)}
    -
    \mathsf A_i\widetilde{\mathsf u}_i^{(k)},
    \qquad
    \mathsf r_{c,i}^{(k)}
    =
    \mathsf P_i^\top
    \mathsf D_i
    \widetilde{\mathsf r}_i^{(k)}.
\]
Here \(\widetilde{\mathsf u}_i^{(k)}\) denotes a local state generated by the MS-DD-ROM trajectory with the trained \(T_{\rm rom}\) and \(S_{\rm rom}\) already deployed. The corresponding target in the training set is computed offline by applying the exact local residual evaluation and the exact coarse projection.
The empirical loss is imposed on the local coarse residual contributions:
\begin{equation}
    \mathcal L_C(\theta_C)
    =
    \frac{1}{|\mathcal D_C|}
    \sum_{\mathcal D_C}
    \left\|
    C_{\rm rom}^{\theta_C}
    \left(
        \widetilde{\mathsf u}_i^{(k)}
    \right)
    -
    \mathsf r_{c,i}^{(k)}
    \right\|_2^2 .
    \label{eq:loss-Crom}
\end{equation}
The training of the coarse residual surrogate is therefore defined by
\[
    (\theta_C)^\star
    \in
    \arg\min_{\theta_C}
    \mathcal L_C(\theta_C).
    \label{eq:Crom-optimal-parameters}
\]
After training, the deployed coarse residual surrogate is denoted by \(C_{\rm rom}^{(\theta_C)^\star}\). When no ambiguity arises, we keep the lighter notation \(C_{\rm rom}^{\theta_C}\).
During the online stage, the coarse right-hand side is assembled from the predicted local contributions,
\[
    \widetilde{\mathsf r}_c^{(k)}
    =
    \sum_{i\in\mathcal I}
    C_{\rm rom}^{(\theta_C)^\star}
    \left(
        \widetilde{\mathsf u}_i^{(k)}
    \right),
\]
and the coarse correction is computed as
\[
    \mathsf A_c\widetilde{\mathsf e}_c^{(k)}
    =
    \widetilde{\mathsf r}_c^{(k)},
    \qquad
    \widetilde{\mathsf u}^{(k)}
    \leftarrow
    \widetilde{\mathsf u}^{(k)}
    +
    \mathsf P\widetilde{\mathsf e}_c^{(k)}.
\]
Thus, \(C_{\rm rom}\) avoids both the fine-scale residual evaluation and the fine-scale residual projection required to form the coarse right-hand side, while retaining the deterministic coarse solve on \(\mathcal V_c\).
The three optimized parameter sets
\[
    (\theta_T)^\star,\qquad
    \theta_S^\star :=
    \left((\theta_S^{\rm rb})^\star,(\theta_S^{\rm c})^\star\right),
    \qquad
    (\theta_C)^\star
\]
define the deployed MS-DD-ROM operators used in the online iteration described
below.

\subsection{Online MS-DD-ROM algorithm}
\label{subsec:online-ddrom}

The online MS-DD-ROM algorithm is obtained by replacing the exact local Neumann/DtN residual evaluation, the local subdomain solve, and the local coarse-residual projection by \(T_{\rm rom}\), \(S_{\rm rom}\), and \(C_{\rm rom}\), respectively. The fine-scale local matrices \(\mathsf A_i\) are used only offline to generate training data. During the online phase, local residuals and local solutions are produced by neural inference, while the coarse solve is retained as a small deterministic algebraic correction.

\begin{algorithm}
\caption{MS-DD-ROM framework for the coupled 3D--1D problem}
\label{alg:SURROGATE_DD}
\begin{algorithmic}[1]
\Require \(T_{\rm rom}\), \(S_{\rm rom}\), \(C_{\rm rom}\);
\(\{\mathsf d_i\}_{i\in\mathcal I}\);
initial states \((\widetilde{\mathsf u}^{(0)},\mathsf u_\Lambda^{(0)})\);
coarse matrix \(\mathsf A_c\).

\For{\(k=1,2,\ldots\) until convergence}

    \Statex
    \Statex \(\triangleright\) \textbf{Coarse correction}

    \For{\(i\in\mathcal I\)}
        \State \(\displaystyle
        \widetilde{\mathsf u}_i^{(k-1)}
        \gets
        \mathsf L_i\widetilde{\mathsf u}^{(k-1)}
        \)
        \Comment{restriction to \(\Omega_i\)}

        \State \(\displaystyle
        \widetilde{\mathsf r}_{c,i}^{(k-1)}
        \gets
        C_{\rm rom}
        \left(
            \widetilde{\mathsf u}_i^{(k-1)}
        \right)
        \)
        \Comment{surrogate coarse-residual contribution}
    \EndFor

    \State \(\displaystyle
    \widetilde{\mathsf r}_c^{(k-1)}
    \gets
    \sum_{i\in\mathcal I}
    \widetilde{\mathsf r}_{c,i}^{(k-1)}
    \)
    \Comment{coarse residual}

    \State \(\displaystyle
    \widetilde{\mathsf e}_c^{(k-1)}
    \gets
    \mathsf A_c^{-1}
    \widetilde{\mathsf r}_c^{(k-1)}
    \)
    \Comment{coarse solve}

    \State \(\displaystyle
    \widetilde{\mathsf u}^{(k-1)}
    \gets
    \widetilde{\mathsf u}^{(k-1)}
    +
    \mathsf P\widetilde{\mathsf e}_c^{(k-1)}
    \)
    \Comment{coarse update}

    \Statex
    \Statex \(\triangleright\) \textbf{Local DD iteration}

    \For{\(i\in\mathcal I\)}
        \State \(\displaystyle
        \widetilde{\mathsf u}_i^{(k-1)}
        \gets
        \mathsf L_i\widetilde{\mathsf u}^{(k-1)}
        \)
        \Comment{restriction to \(\Omega_i\)}

        \State \(\displaystyle
        \widetilde{\mathsf f}_i^{(k-1)}
        \gets
        T_{\rm rom}
        \left(
            \widetilde{\mathsf u}_i^{(k-1)}
        \right)
        \)
        \Comment{surrogate Robin datum}
    \EndFor

    \For{\(i\in\mathcal I\)}
        \State \(\displaystyle
        \widetilde{\mathsf f}_{{\rm agg},i}^{(k-1)}
        \gets
        \sum_{j\in\mathcal N(i)}
        \mathsf L_i\mathsf L_j^\top
        \widetilde{\mathsf f}_j^{(k-1)}
        \)
        \Comment{interface aggregation}

        \State \(\displaystyle
        \mathsf u_{\Lambda\to h,i}^{(k-1)}
        \gets
        \mathsf E^\Lambda_{h,i}
        \mathsf u_{\Lambda,i}^{(k-1)}
        \)
        \Comment{embedded-to-bulk interpolation}

        \State \(\displaystyle
        \widetilde{\mathsf u}_i^{(k)}
        \gets
        S_{\rm rom}
        \left(
            \mathsf d_i,
            \mathsf u_{\Lambda\to h,i}^{(k-1)},
            \widetilde{\mathsf f}_{{\rm agg},i}^{(k-1)}
        \right)
        \)
        \Comment{surrogate local 3D solve}
    \EndFor

    \Statex
    \Statex \(\triangleright\) \textbf{Global reconstruction and 1D update}

    \State Reconstruct
    \(\widetilde{\mathsf u}^{(k)}\) from
    \(\{\widetilde{\mathsf u}_i^{(k)}\}_{i\in\mathcal I}\)
    \Comment{global 3D field}

    \State \(\displaystyle
    \mathsf u_\Lambda^{(k)}
    \gets
    \mathfrak s_\Lambda
    \left(
        \widetilde{\mathsf u}^{(k)},\mu
    \right)
    \)
    \Comment{global 1D solve}

\EndFor

\State \Return
\((\widetilde{\mathsf u}^{(k)},\mathsf u_\Lambda^{(k)})\)

\end{algorithmic}
\end{algorithm}

The computational advantage of the presented algorithm is that all fine-scale matrix-dependent operations are replaced by neural inference. The only algebraic solve retained online is the coarse correction, whose dimension is \(\dim(\mathcal V_c)\) and is therefore independent of the local fine-scale discretization.

The resulting MS-DD-ROM framework is therefore designed to be fast, accurate, and generalizable: fast, because online local solves and residual evaluations are replaced by neural evaluations; accurate, because the learned maps approximate the local high-fidelity operators in norms aligned with the DD stability estimate; and generalizable, because the same reference surrogate is trained on a localized parameter space representative of all subdomain configurations encountered by the global algorithm.

\section{Validation of the MS-DD-ROM framework}
\label{sec:oxygen_model}

We now assess the MS-DD-ROM framework on the mixed-dimensional oxygen-perfusion model introduced in Section~\ref{sec:3D1D_problem_localization}. The purpose of this section is not only to evaluate the accuracy of the surrogate approximation, but also to verify the main methodological properties developed in the previous sections: reuse of a common local surrogate across unseen microstructures, stability of the surrogate-perturbed DD iteration, effectiveness of the two-level correction under increasing problem size, and reduction of the online computational cost.
%We apply MS-DD-ROM to a challenging problem describing oxygen-perfusion using the mixed-dimensional formulation introduced in Section~\ref{sec:3D1D_problem_localization}.
The 1D manifold $\Lambda$ models the blood vessels, while the surrounding 3D region $\Omega$ represents the interstitial tissue through which oxygen diffuses. We use the problem definition previously proposed by Vidotto \textit{et al}~\cite{vidotto2019hybrid}:
\begin{equation}
    \begin{cases}
        -\nabla \cdot \left( \rho_{\text{int}} \dfrac{K_{t}}{\mu_{\text{int}}} \nabla p^{t} \right)
        = (2\pi \epsilon)L_{\text{cap}} \rho_{\text{int}}
        \left( p^V - \overline{\mathcal{T}}_\Lambda p^{t} - (\pi_{p} - \pi_{\text{int}}) \right) \delta_{\Lambda}
        & \text{in}\ \Omega, \\[10pt]
        -\dfrac{d}{ds} \left( \rho_{\text{bl}} (\pi \epsilon^{2}) \dfrac{K_{V}}{\mu_{\text{bl}}} \dfrac{dp^{V}}{ds} \right)
        = (2\pi \epsilon) L_{\text{cap}} \rho_{\text{int}}
        \left( \overline{\mathcal{T}}_\Lambda p^{t} - p^V + (\pi_{p} - \pi_{\text{int}}) \right)
        & \text{in}\ \Lambda, \\[10pt]
        -\rho_{\text{int}} \dfrac{K_{t}}{\mu_{\text{int}}} \nabla p^{t} \cdot \mathbf{n} = \beta(p^{t} - p_0)
        & \text{on}\ \partial\Omega, \\[10pt]
        p^{V} = p_\text{dir}^{V} & \text{on}\ \partial\Lambda,
    \end{cases}
\label{eq:OxModelVid}
\end{equation}
where the unknowns are $(p^t, p^V)$, denoting respectively the \textit{tissue} and \textit{vascular} pressures.  
The physical constants involved in the model are summarized in Table~\ref{tab:parameters}. A visualization of the solution of the problem can be found in Fig~\ref{fig:show_solution}.
\begin{figure}[h!]
    \centering
    \includegraphics[width=0.80\linewidth]{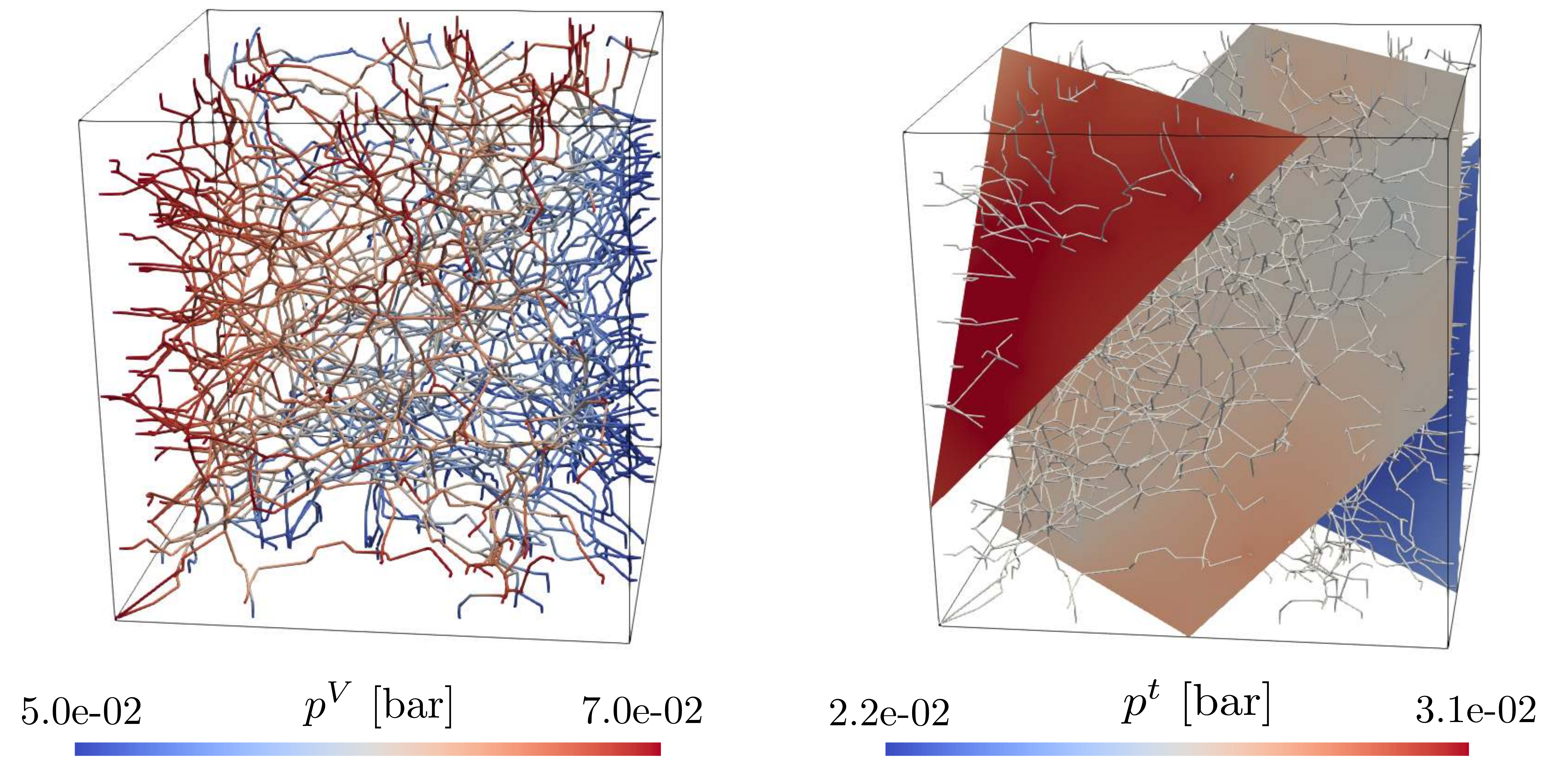}
\caption{Visualization of the pressure fields for the coupled 3D--1D perfusion problem. 
\textbf{Left:} Vascular pressure $p^V$. \textbf{Right:} Corresponding slices of the tissue pressure field, $p^t$, shown together with the embedded vasculature. Both pressure fields are expressed in bar.}
    \label{fig:show_solution}
\end{figure}
\begin{table}[h!]
    \centering
    \small
    \begin{tabular}{|l|l|l|}
        \hline
        \textbf{Parameter} & \textbf{Value} & \textbf{Description} \\ \hline
        $\epsilon$ & $1 \times 10^{-2}$ & Vessel radius (mm) \\ \hline
        $\mu_{\text{int}}$ & $1.3 \times 10^{-8}$ & Plasma viscosity (bar·s) \\ \hline
        $\mu_{\text{bl}}$ & $5 \times 10^{-9}$ & Blood viscosity (bar·s) \\ \hline
        $K_t$ & $1 \times 10^{-12}$ & Tissue permeability (mm\(^2\)) \\ \hline
        $K_V$ & $\epsilon^2/8 = 1.25 \times 10^{-5}$ & Vessel permeability (mm\(^2\)) \\ \hline
        $\rho_{\text{int}}$ & $1 \times 10^{-6}$ & Tissue fluid density (kg/mm\(^3\)) \\ \hline
        $\rho_{\text{bl}}$ & $1.03 \times 10^{-6}$ & Blood density (kg/mm\(^3\)) \\ \hline
        $L_{\text{cap}}$ & $1 \times 10^{-4}$ & Capillary wall conductivity (mm/(bar·s)) \\ \hline
        $\pi_p - \pi_{\text{int}}$ & $0.03334$ & Osmotic pressure difference (bar) \\ \hline
        $\beta$ & $0$ & Boundary exchange coefficient (kg/(bar·mm\(^2\)·s)) \\ \hline
        $p_{D,\text{in}}^V$ & $0.07$ & Inlet pressure (bar) \\ \hline
        $p_{D,\text{out}}^V$ & $0.05$ & Outlet pressure (bar) \\ \hline
    \end{tabular}
    \caption{Physical parameters for the 3D--1D oxygen perfusion model. }
    \label{tab:parameters}
\end{table}

\subsection{Computational setting and training data}
\label{sec:computational_setting}

The high-fidelity reference configuration is defined on
\(
\Omega=[0,3]^3\ {\rm mm}^3,
\)
partitioned into \(3\times3\times3\) congruent subdomains of side length \(1\,{\rm mm}\). Each subdomain is identified with the reference configuration
\(
\Omega_\star=[0,1]^3\ {\rm mm}^3.
\)
The local three-dimensional problem is discretized using continuous piecewise linear finite elements on a structured tetrahedral mesh obtained from a \(20\times20\times20\) Cartesian partition. This corresponds to a mesh size
\(
h=50\,\mu{\rm m}
\)
and approximately \(9.3\times10^3\) degrees of freedom per local problem.

The geometric variability is represented by a collection of \(75\) synthetic vascular networks
\(
\{\Lambda^{(g)}\}_{g=1}^{75},
\)
generated from the same stochastic construction based on three-dimensional Voronoi tessellations and shortest-path extraction. The dataset is divided into \(65\) training, \(5\) validation, and \(5\) test geometries. Each global network contains approximately \(2.0\times10^4\) one-dimensional degrees of freedom and \(2.5\times10^4\) edges on average. The test geometries are never used during training and are employed to assess generalization with respect to unseen embedded microstructures.

The training samples are generated according to the cascaded procedure of Section~\ref{sec:cascade_training}. Data are collected over graph realizations, subdomains, and selected DD iterations. For each surrogate, the resulting dataset contains
\(
65\times27\times5=8775
\)
local samples. The recorded quantities are
\[
\left(
\mathsf d_i,
\mathsf u_i^{(k)},
\mathsf u_{\Lambda\to h,i}^{(k-1)},
\mathsf f_{\mathrm{agg},i}^{(k-1)}
\right),
\qquad
k\in\{1,\ldots,5\}.
\]
The three training sets are not generated from the same algorithmic distribution. The transmission surrogate is trained on MS-DD-FOM trajectories, the local solution surrogate on trajectories with the trained transmission surrogate already active, and the coarse-residual surrogate on trajectories generated with both transmission and local-solution surrogates deployed.

To account for the non-stationary character of the DD trajectories, iteration-dependent surrogate instances are used. A single transmission model is trained on samples from all selected iterations. For \(S_{\rm rom}\), separate models are used for the first DD iteration and for subsequent iterations, while \(C_{\rm rom}\) uses three regimes corresponding to the first, second, and later iterations. During online execution, the surrogate associated with the current iteration regime is selected.

All surrogate models are implemented in \texttt{PyTorch}~\cite{paszke2019pytorch} and trained using the \texttt{L-BFGS} optimizer~\cite{liu1989limited} with learning rate equal to \(1.0\). The network architectures and the corresponding training parameters are summarized in Table~\ref{tab:rom_arch_summary}. Training was performed on a Lenovo Legion i9 laptop equipped with an Intel Core i9 processor and an NVIDIA GeForce RTX 4090 Laptop GPU with \(16\,\mathrm{GB}\) of GDDR6 memory. The largest individual training run required approximately \(600\,\mathrm{s}\).

The predictive performance of the three surrogate operators is reported in Table~\ref{tab:rom_performance}. For the transmission and local solution surrogates, \(T_{\rm rom}\) and \(S_{\rm rom}\), the error is measured as the relative discrete \(L^2(\Omega_\star)\)-error averaged over the corresponding dataset split. Since the output of \(C_{\rm rom}\) is a coarse residual vector rather than a finite element field, its accuracy is quantified by the relative Euclidean mean-square error.

\begin{table}[h!]
\centering
\small
\renewcommand{\arraystretch}{1.10}
\setlength{\tabcolsep}{5pt}
\begin{tabular}{lll}
\hline
\textbf{Surrogate} & \textbf{Quantity} & \textbf{Description} \\
\hline
\(T_{\rm rom}\)
& Architecture
& \( {\rm MINN}_{0.3}+{\rm MINN}_{0.2}+{\rm MINN}_{0.15}\) \\
& Trainable parameters
& \(1{,}389{,}143\) \\
& Training setup
& 200 epochs, L-BFGS optimizer, learning rate \(1.0\) \\
\hline
\(S_{\rm rom}\) (POD-MINN)
& Architecture
& \(3\) MINN encoders \(+\) dense decoder, \(100\) POD modes \\
& Trainable parameters
& \(3{,}255{,}806\)
\\
& Training setup
& 100 epochs, L-BFGS optimizer, learning rate \(1.0\) \\
\hline
\(S_{\rm rom}\) (closure)
& Architecture
& \( {\rm MINN}_{0.2}\) \\
& Trainable parameters
& \(1{,}908{,}576\) \\
& Training setup
& 25 epochs, L-BFGS optimizer, learning rate \(1.0\) \\
\hline
\(C_{\rm rom}\)
& Architecture
& \( {\rm MINN}_{0.3}+{\rm MINN}_{0.3}+{\rm Dense}_{256}+{\rm Dense}_{64}+{\rm Dense}_{8}\) \\
& Trainable parameters
& \(1{,}952{,}434\) \\
& Training setup
& 50 epochs, L-BFGS optimizer, learning rate \(1.0\) \\
\hline
\end{tabular}
\caption{Summary of the surrogate architectures and training configurations. Here,
\({\rm Local}_{r}\) denotes a mesh-informed local layer with interaction radius
\(r\), and \({\rm Dense}_{m}\) denotes a dense layer with output dimension \(m\).
The local solution surrogate \(S_{\rm rom}\) is implemented as a POD-MINN+
architecture with a closure correction.}
\label{tab:rom_arch_summary}
\end{table}

\begin{table}[h!]
\centering
\small
\renewcommand{\arraystretch}{1.2}
\setlength{\tabcolsep}{10pt}
\begin{tabular}{lccc}
\hline
\textbf{Model} & \textbf{Training [\%]} & \textbf{Validation [\%]} & \textbf{Test [\%]} \\
\hline
\(T_{\rm rom}\) (Transmission) &
1.03 & 1.02 & 1.04 \\
\(S_{\rm rom}\) (Local solution) &
4.55 & 5.06 & 5.58 \\
\(C_{\rm rom}\) (Coarse residual) &
7.33 & 8.82 & 9.47 \\
\hline
\end{tabular}
\caption{Training, validation, and test errors of the three surrogate maps. Errors for \(T_{\rm rom}\) and \(S_{\rm rom}\) are relative discrete \(L^2(\Omega_\star)\)-errors, while errors for \(C_{\rm rom}\) are relative Euclidean mean-square errors on the coarse residual vectors.}
\label{tab:rom_performance}
\end{table}

\subsection{Global accuracy and stability of the MS-DD-ROM iteration}
\label{sec:global_accuracy}

Having assessed the predictive accuracy of the individual surrogate operators, we now evaluate their coupled deployment within the complete MS-DD-ROM iteration. The purpose of this experiment is to determine whether the learned transmission, local-solution, and coarse-residual maps preserve the stability and interface consistency of the underlying two-level MS-DD-FOM on vascular geometries not used during training. The reported errors concern the three-dimensional component, which is the computationally dominant part of the coupled problem and the only component replaced by the MS-DD-ROM. The one-dimensional vascular problem remains an exact global solve within the staggered iteration. Its convergence was monitored in all experiments and remained consistent with that of the coupled three-dimensional iteration.

We consider ten unseen vascular networks and compare the MS-DD-ROM iterates with both the corresponding MS-DD-FOM iterates and a reference finite element solution \(u^\star\). The local linear systems of the MS-DD-FOM are solved by preconditioned conjugate gradients using PETSc/Hypre AMG, with relative tolerance \(10^{-15}\), so that the reported discrepancies are dominated by the surrogate approximation rather than by the local iterative solver.

The comparison is based on three complementary quantities. The global field accuracy is measured by the relative \(L^2\)-error
\[
E_{L^2}^{(k)}
:=
\frac{\|u_h^{(k)}-u^\star\|_{L^2(\Omega)}}
{\|u^\star\|_{L^2(\Omega)}}.
\]
Since the MS-DD-ROM reconstructs the global field from independently predicted local states, interface consistency is monitored through
\[
J_\Gamma^{(k)}
:=
\left(
\sum_{\Gamma_{ij}}
\left\|
u_{h,i}^{(k)}-u_{h,j}^{(k)}
\right\|_{L^2(\Gamma_{ij})}^2
\right)^{1/2}.
\]
To measure local gradient errors together with the loss of inter-subdomain continuity, we also introduce the relative broken energy-type error
\[
E_{\rm br}^{(k)}
:=
\frac{
\left(
\sum_{i\in\mathcal I}
\|u_{h,i}^{(k)}-u^\star\|_{H^1(\Omega_i)}^2
+
\sum_{\Gamma_{ij}}
h_{ij}^{-1}
\|u_{h,i}^{(k)}-u_{h,j}^{(k)}\|_{L^2(\Gamma_{ij})}^2
\right)^{1/2}
}{
\|u^\star\|_{H^1(\Omega)}
}.
\]
This broken norm is commonly used to combine elementwise energy errors and interface mismatches in nonconforming discretizations.

\begin{figure}[h!]
    \centering
    \includegraphics[width=0.99\linewidth]
    {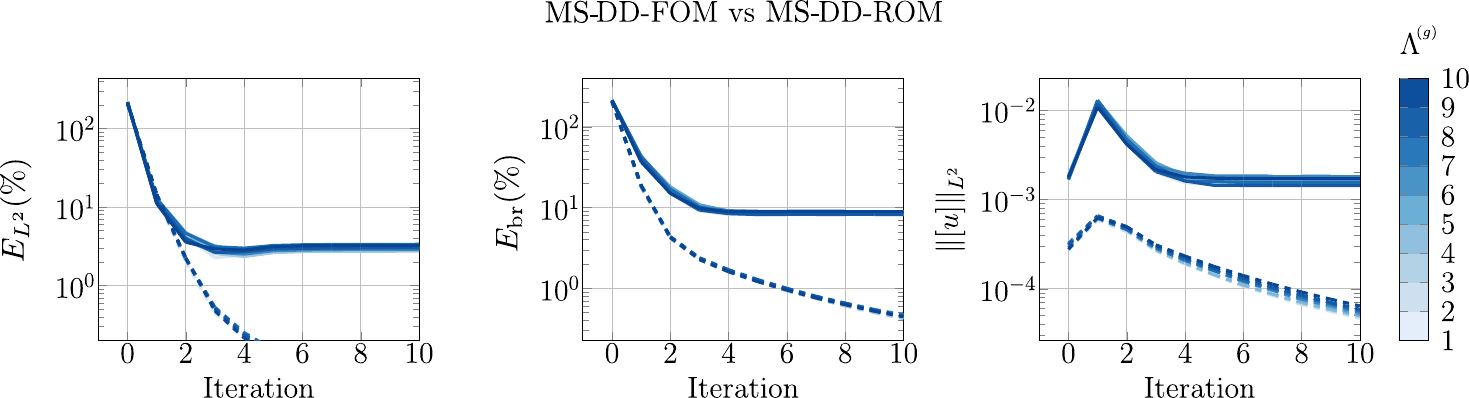}
\caption{Convergence curves of the MS-DD-FOM and MS-DD-ROM algorithms for the oxygen perfusion problem~\eqref{eq:OxModelVid}. Dashed lines denote the convergence histories obtained with the exact RR local solves for ten unseen one-dimensional geometries \(\Lambda^{(g)}\). Solid lines denote the corresponding convergence histories obtained by replacing the high-fidelity local operations with the surrogate RR maps. The reported errors include the relative \(L^2\)-error, the relative broken-energy error, and the interface jump norm.}
   \label{fig:DD_iteer_images}
\end{figure}

\begin{table}[h!]
\centering
\small
\renewcommand{\arraystretch}{1.2}
\setlength{\tabcolsep}{5.5pt}
\resizebox{\linewidth}{!}{%
\begin{tabular}{c|
cccc|cccc}
\hline
\multirow{2}{*}{$k$}
& \multicolumn{4}{c|}{\textbf{MS-DD-FOM}}
& \multicolumn{4}{c}{\textbf{MS-DD-ROM}} \\
\cline{2-9}
& $E_{L^2}$~[\%] & $E_{\mathrm{br}}$~[\%] & $\|[\![u]\!]\|_{L^2(\Gamma)}$ & $\Delta$~[\%]
& $E_{L^2}$~[\%] & $E_{\mathrm{br}}$~[\%] & $\|[\![u]\!]\|_{L^2(\Gamma)}$ & $\Delta$~[\%] \\
\hline
0
& $2.09\mathrm{E}{+2}$ & $2.07\mathrm{E}{+2}$ & $3.03\mathrm{E}{-4}$ & $1.14\mathrm{E}{+0}$
& $2.16\mathrm{E}{+2}$ & $2.14\mathrm{E}{+2}$ & $1.78\mathrm{E}{-3}$ & $3.81\mathrm{E}{+1}$ \\
2
& $2.05\mathrm{E}{+0}$ & $4.12\mathrm{E}{+0}$ & $4.61\mathrm{E}{-4}$ & $8.00\mathrm{E}{-1}$
& $4.38\mathrm{E}{+0}$ & $1.63\mathrm{E}{+1}$ & $4.54\mathrm{E}{-3}$ & $1.17\mathrm{E}{+1}$ \\
4
& $2.39\mathrm{E}{-1}$ & $1.63\mathrm{E}{+0}$ & $2.08\mathrm{E}{-4}$ & $1.59\mathrm{E}{-1}$
& $2.72\mathrm{E}{+0}$ & $8.85\mathrm{E}{+0}$ & $1.76\mathrm{E}{-3}$ & $1.92\mathrm{E}{+0}$ \\
6
& $1.01\mathrm{E}{-1}$ & $9.63\mathrm{E}{-1}$ & $1.23\mathrm{E}{-4}$ & $8.34\mathrm{E}{-2}$
& $3.08\mathrm{E}{+0}$ & $8.84\mathrm{E}{+0}$ & $1.68\mathrm{E}{-3}$ & $1.24\mathrm{E}{+0}$ \\
8
& $5.15\mathrm{E}{-2}$ & $6.30\mathrm{E}{-1}$ & $7.96\mathrm{E}{-5}$ & $5.07\mathrm{E}{-2}$
& $3.13\mathrm{E}{+0}$ & $8.85\mathrm{E}{+0}$ & $1.68\mathrm{E}{-3}$ & $1.22\mathrm{E}{+0}$ \\
\hline
\end{tabular}
} % end resizebox
\caption{\small Convergence history comparing the three-dimensional MS-DD-FOM and MS-DD-ROM solutions for the representative global one-dimensional graph \(\Lambda_1\). At each iteration \(k\), the relative \(L^2\)-error \(E_{L^2}\), the broken-energy error \(E_{\mathrm{br}}\), and the interface jump norm \(\|[\![u]\!]\|_{L^2(\Gamma)}\) are reported for the three-dimensional solution. The corresponding relative increments \(\Delta\) are defined between two consecutive iterations as \( \Delta = {\left\|u^{(k)}-u^{(k-1)}\right\|}/{\|u^{(k-1)}}\| \) and are reported in percentage.} \label{tab:convergence_1D_3D_surrogate}
\end{table}

Figure~\ref{fig:DD_iteer_images} reports the convergence histories over the ten unseen vascular geometries, while Table~\ref{tab:convergence_1D_3D_surrogate} provides detailed values for one representative configuration.

The exact MS-DD-FOM converges toward the conforming reference solution, whereas the MS-DD-ROM enters a stable error plateau after approximately four iterations. Importantly, the surrogate iteration does not exhibit error growth or loss of interface stability. This behavior is consistent with Proposition~\ref{prop:perturbed_dd}: the learned local operators perturb the exact contraction, so that the iterates approach a neighborhood of the fixed point rather than the fixed point itself. The observed plateau therefore provides numerical evidence of the error-floor mechanism predicted by the abstract analysis.
For the representative test geometry, the plateau corresponds to a relative \(L^2\)-error of approximately \(3.1\%\), a broken energy-type error of approximately \(8.8\%\), and an interface jump of order \(10^{-3}\).

The spatial comparison in Figure~\ref{fig:solution_comparison} shows that the MS-DD-ROM reproduces the global pressure pattern and the localized variations induced by the embedded vascular network. The largest discrepancies are concentrated near the artificial subdomain interfaces, where pointwise relative errors reach approximately \(5\%\). This localized behavior aligns with the structure of \(S_{\rm rom}\): Robin data are supplied as input fields, but interface continuity and flux balance are not enforced exactly in its output. The result therefore identifies interface-aware constraints as a natural direction for improving the surrogate while preserving the present non-intrusive formulation.

\begin{figure}[h!]
    \centering
    \includegraphics[width=0.95\linewidth]{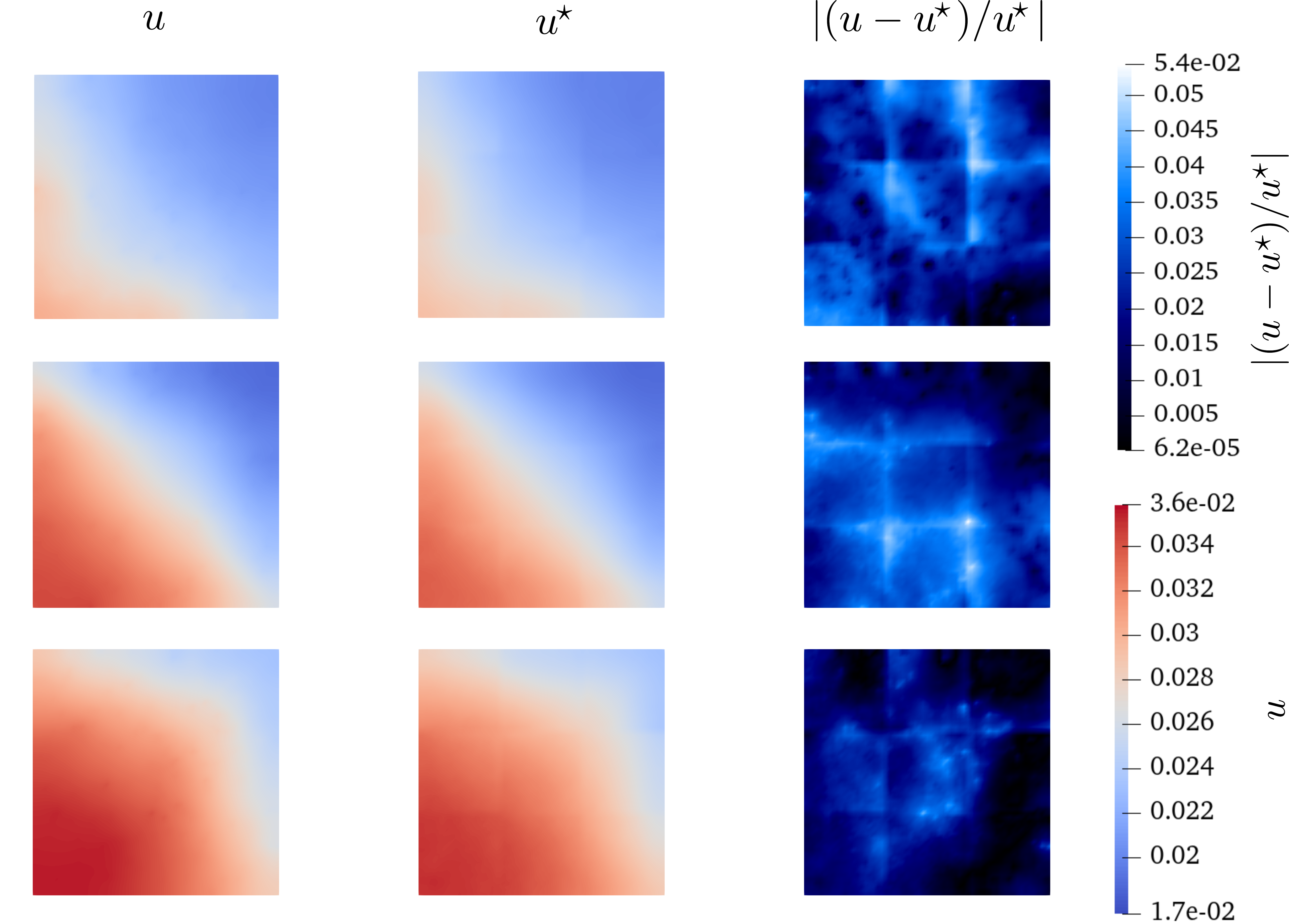}
\caption{Tissue pressure solution \(u \equiv p^t \) on the domain \(\Omega=[0\,\mathrm{mm},3\,\mathrm{mm}]^3\). From top to bottom, the rows correspond to slices taken in the upper, central, and lower regions of the domain along the \(z\)-direction. \textbf{Left column:} reference solution \(u^\star\). \textbf{Center column:} MS-DD-ROM solution at stagnation. \textbf{Right column:} pointwise relative error \(\left|(u-u^\star)/u^\star\right|\) between the MS-DD-ROM and reference solutions.}
    \label{fig:solution_comparison}
\end{figure}

\subsubsection{Weak scalability of the MS-DD-ROM algorithm}

We next investigate whether the MS-DD-ROM preserves the weak-scaling behavior of the underlying two-level domain decomposition method. The key question is whether surrogates trained on a fixed reference subdomain can be reused without retraining when the global domain is enlarged by increasing the number of subdomains.
To this end, the physical size and fine-scale discretization of each local subdomain are kept fixed, while the global domain is enlarged. Each subdomain has volume
\(
|\Omega_i|=1\,\mathrm{mm}^3
\)
and is discretized exactly as in the training configuration. For
\(
n\in\{3,4,5\}
\),
we consider Cartesian decompositions with
\(
N_{\rm sub}=n^3\in\{27,64,125\},
\)
corresponding to approximately
\(
2.5\times10^5,\
5.9\times10^5,\
1.16\times10^6
\)
three-dimensional degrees of freedom, respectively.
The local surrogate architectures, trained parameters, reference mesh, and coarse-space construction are unchanged across the three configurations. Only the number of instantiated subdomains and the size of the global vascular network increase. The embedded vascular networks are sampled from the same generative model used for the training set, but the resulting global graphs and decompositions are not present in the offline data

\begin{figure}[t]
    \centering
    \includegraphics[width=0.95\linewidth]{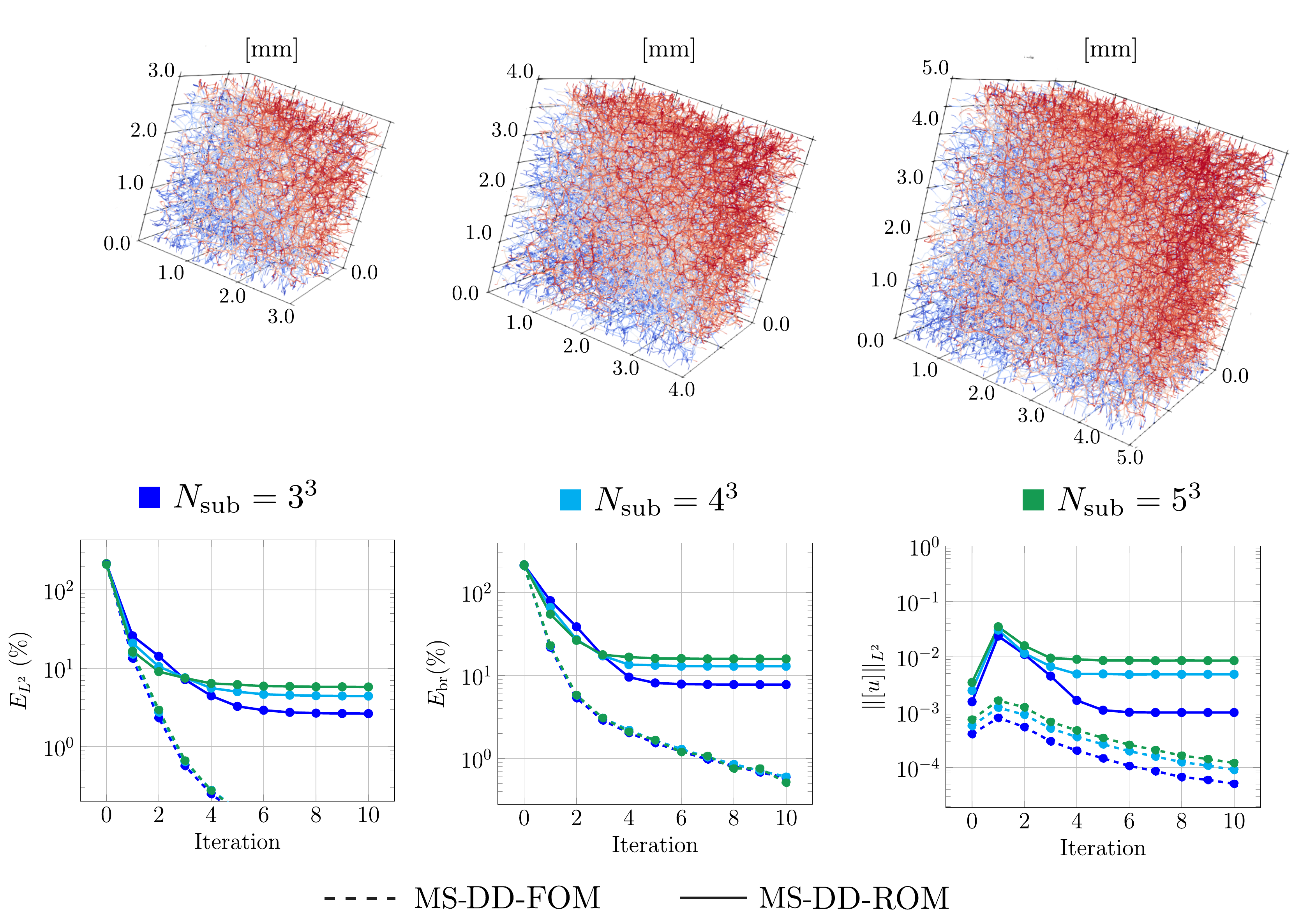}
 \caption{Weak-scaling study for the MS-DD-ROM algorithm applied to the oxygen perfusion problem~\eqref{eq:OxModelVid}. From left to right, both the top and bottom rows refer to the \(3^3\), \(4^3\), and \(5^3\) subdomain configurations, respectively. \textbf{Top row:} scaled representations, in millimeters, of the embedded one-dimensional geometries inside the three-dimensional domains, corresponding to global volumes of \(27\,\mathrm{mm}^3\), \(64\,\mathrm{mm}^3\), and \(125\,\mathrm{mm}^3\). \textbf{Bottom row:} convergence curves of the MS-DD-FOM and MS-DD-ROM algorithms. Dashed lines denote the convergence histories obtained with the exact RR local solves for ten different one-dimensional geometries \(\Lambda_i\), while solid lines denote the corresponding histories obtained by replacing the high-fidelity local operations with the surrogate RR maps. The reported quantities are the relative \(L^2\)-error, the relative broken-energy error, and the interface jump norm.} \label{fig:ddromscaling}
\end{figure}

Here, weak scalability is understood in an algorithmic sense: the local computational problem remains fixed and we assess whether the convergence behavior and global accuracy remain controlled as the number of subdomains increases. Parallel wall-clock scalability is not considered in this experiment.

Figure~\ref{fig:ddromscaling} reports the convergence histories of the MS-DD-FOM and MS-DD-ROM for the three decomposition sizes, while Table~\ref{tab:convergence_3d_dofs} provides representative quantitative values at selected iterations. The nearly overlapping MS-DD-FOM curves confirm that the two-level Robin--Robin method preserves its contraction behavior as the global domain grows. This demonstrates the effectiveness of the coarse correction in transmitting global information across increasingly large decompositions.

The MS-DD-ROM exhibits the same qualitative stability mechanism. In all three configurations, the surrogate iteration decreases rapidly during the initial iterations and subsequently reaches a bounded error plateau, without evidence of instability or deterioration of the convergence rate. As expected, we see that the two-level DD mechanism controls the iterative propagation of the error, while the surrogate approximation determines the asymptotic neighborhood of the exact fixed point.

The final accuracy nevertheless deteriorates moderately with the number of subdomains. The relative \(L^2\)-error increases from approximately \(2.7\%\) for \(27\) subdomains to \(5.8\%\) for \(125\) subdomains, while the broken energy-type error increases from approximately \(7.8\%\) to \(15.8\%\). This degradation may result both from the accumulation of local surrogate errors over a larger decomposition and from localized DD states that are progressively less represented in the offline training distribution.

Overall, the experiment demonstrates algorithmic weak scalability of the MS-DD-ROM over the tested range: the same reference surrogates can be deployed on global problems exceeding one million degrees of freedom without loss of iterative stability. The growth of the error plateau shows, however, that scalability in accuracy ultimately depends on the coverage of the localized parameter space. This observation suggests adaptive enrichment strategies in which additional high-fidelity local samples are generated only for the subdomains and DD states responsible for the largest residual or interface errors.

\begin{table}[h!]
\centering
\small
\renewcommand{\arraystretch}{1.05}
\setlength{\tabcolsep}{6pt}
\resizebox{0.75\linewidth}{!}{%
\begin{tabular}{c|c|cccc}
\hline\small
\textbf{Problem size} & \(\boldsymbol{k}\)
& \(E_{L^2}\)~[\%]
& \(E_{\mathrm{br}}\)~[\%]
& \(\|[u]\|_{L^2(\Gamma)}\)
& \(\Delta\)~[\%] \\
\hline

\multirow{5}{*}{\begin{tabular}{c}$N_\mathrm{sub}=27$; \\ \(2.50\mathrm{E}{+5}\) DoFs\end{tabular}}
& 0 & \(2.17\mathrm{E}{+2}\) & \(2.14\mathrm{E}{+2}\) & \(1.54\mathrm{E}{-3}\) & \(2.54\mathrm{E}{+1}\) \\
& 2 & \(1.42\mathrm{E}{+1}\) & \(3.84\mathrm{E}{+1}\) & \(1.10\mathrm{E}{-2}\) & \(2.17\mathrm{E}{+1}\) \\
& 4 & \(4.44\mathrm{E}{+0}\) & \(9.52\mathrm{E}{+0}\) & \(1.63\mathrm{E}{-3}\) & \(4.12\mathrm{E}{+0}\) \\
& 6 & \(2.91\mathrm{E}{+0}\) & \(7.85\mathrm{E}{+0}\) & \(9.95\mathrm{E}{-4}\) & \(9.75\mathrm{E}{-1}\) \\
& 8 & \(2.67\mathrm{E}{+0}\) & \(7.76\mathrm{E}{+0}\) & \(9.87\mathrm{E}{-4}\) & \(6.00\mathrm{E}{-1}\) \\
\hline

\multirow{5}{*}{\begin{tabular}{c}$N_\mathrm{sub}=64$; \\ \(5.93\mathrm{E}{+5}\) DoFs\end{tabular}}
& 0 & \(2.16\mathrm{E}{+2}\) & \(2.13\mathrm{E}{+2}\) & \(2.46\mathrm{E}{-3}\) & \(2.45\mathrm{E}{+1}\) \\
& 2 & \(1.05\mathrm{E}{+1}\) & \(2.70\mathrm{E}{+1}\) & \(1.16\mathrm{E}{-2}\) & \(1.37\mathrm{E}{+1}\) \\
& 4 & \(5.57\mathrm{E}{+0}\) & \(1.35\mathrm{E}{+1}\) & \(4.88\mathrm{E}{-3}\) & \(4.31\mathrm{E}{+0}\) \\
& 6 & \(4.65\mathrm{E}{+0}\) & \(1.29\mathrm{E}{+1}\) & \(4.79\mathrm{E}{-3}\) & \(3.16\mathrm{E}{+0}\) \\
& 8 & \(4.46\mathrm{E}{+0}\) & \(1.29\mathrm{E}{+1}\) & \(4.83\mathrm{E}{-3}\) & \(3.01\mathrm{E}{+0}\) \\
\hline

\multirow{5}{*}{\begin{tabular}{c}$N_\mathrm{sub}=125$; \\ \(1.16\mathrm{E}{+6}\) DoFs\end{tabular}}
& 0 & \(2.16\mathrm{E}{+2}\) & \(2.14\mathrm{E}{+2}\) & \(3.45\mathrm{E}{-3}\) & \(2.14\mathrm{E}{+1}\) \\
& 2 & \(9.06\mathrm{E}{+0}\) & \(2.64\mathrm{E}{+1}\) & \(1.58\mathrm{E}{-2}\) & \(1.10\mathrm{E}{+1}\) \\
& 4 & \(6.40\mathrm{E}{+0}\) & \(1.66\mathrm{E}{+1}\) & \(9.02\mathrm{E}{-3}\) & \(4.69\mathrm{E}{+0}\) \\
& 6 & \(5.91\mathrm{E}{+0}\) & \(1.59\mathrm{E}{+1}\) & \(8.61\mathrm{E}{-3}\) & \(4.50\mathrm{E}{+0}\) \\
& 8 & \(5.79\mathrm{E}{+0}\) & \(1.58\mathrm{E}{+1}\) & \(8.55\mathrm{E}{-3}\) & \(4.46\mathrm{E}{+0}\) \\
\hline
\end{tabular}%
}
\caption{\small Convergence history for the three-dimensional MS-DD-ROM solutions with \(3^3\), \(4^3\), and \(5^3\) subdomain configurations. At each iteration \(k\), the relative \(L^2\)-error \(E_{L^2}\), the broken-energy error \(E_{\mathrm{br}}\), and the interface jump norm \(\|[\![u]\!]\|_{L^2(\Gamma)}\) are reported for the three-dimensional solution. The corresponding relative increments \(\Delta\) are defined between two consecutive iterations as \( \Delta = {\left\|u^{(k)}-u^{(k-1)}\right\|}/{\left\|u^{(k-1)}\right\|} \) and are reported in percentage.}
\label{tab:convergence_3d_dofs}
\end{table}

\subsubsection{Online computational performance}
\label{sec:online_performance}

We finally assess the online computational cost of the MS-DD-ROM relative to MS-DD-FOM. Since the total execution time of a domain decomposition solver depends strongly on the parallel implementation and hardware architecture, we focus here on the principal operations performed on each subdomain. This comparison isolates the computational savings produced by replacing the fine-scale local operators with neural surrogates.

The measured operations comprise two parameter-dependent setup costs and three iteration-dependent operations. The former consist of \textit{(i)} local matrix assembly and \textit{(ii)} preconditioner construction; the latter are \textit{(iii)} evaluation of the local coarse-residual contribution, \textit{(iv)} computation of the Robin transmission datum, and \textit{(v)} solution of the local subproblem. Their execution frequencies and average costs are reported in Table~\ref{tab:execution_frequency}, while Figure~\ref{fig:DD_time_convergence} provides a logarithmic comparison.

The main computational advantage of the MS-DD-ROM is the elimination of the parameter-dependent fine-scale setup. In the MS-DD-FOM, assembling the local operator requires approximately \(646\,\mathrm{ms}\) per subdomain and dominates the measured cost, while preconditioner setup requires an additional \(49\,\mathrm{ms}\). Neither operation is required online by the MS-DD-ROM, because the local matrices and preconditioners are used only during offline data generation.

The elimination of assembly is particularly important in many-query settings involving changes in the embedded vascular geometry. In the MS-DD-FOM, each new parametrization generally requires reconstructing the local coupling operators and updating the associated solver setup. Once trained, the MS-DD-ROM replaces these repeated setup costs by evaluations of the same reusable reference surrogates.
\begin{figure}[h!]
    \centering
    \includegraphics[width=0.65\linewidth]{ 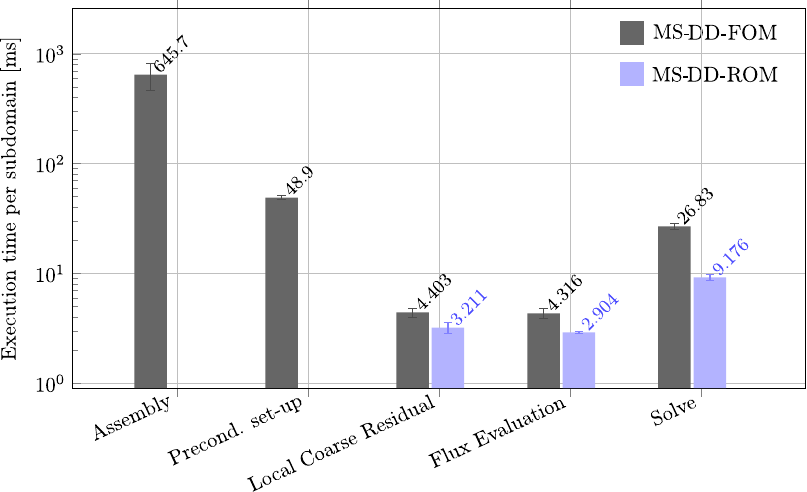}
\caption{Average execution time per subdomain, reported in logarithmic scale, for the main computational operations: matrix assembly, preconditioner setup, flux evaluation, and subdomain solve. Dark bars correspond to the MS-DD-FOM algorithm, while light blue bars indicate the surrogate MS-DD-ROM algorithm. }
\label{fig:DD_time_convergence}
\end{figure}

\begin{table}[h!]
\centering
\small
\renewcommand{\arraystretch}{1.2}
\setlength{\tabcolsep}{8pt}
\resizebox{\linewidth}{!}{%
\begin{tabular}{l|cc|cc}
\hline
\multirow{2}{*}{\textbf{Operation}} 
& \multicolumn{2}{c|}{\textbf{MS-DD-FOM}} 
& \multicolumn{2}{c}{\textbf{MS-DD-ROM}} \\
\cline{2-5}
& \textbf{Frequency} & \textbf{Time  [ms]} 
& \textbf{Frequency} & \textbf{Time [ms]} \\
\hline
Assembly              & once per param. change & 645.7 \(\pm (190, \;180)\)  & not required & - \\
Prec. set-up & once per param. change & 48.90 \(\pm (2.60, \; 2.00)\) & not required & - \\
Coarse Res.       & once per iter. & 4.403 \(\pm (0.90, \; 0.40)\) & once per iter. & 3.211 \(\pm (0.57, \; 0.40)\) \\
Robin Flux       & once per iter. & 4.316 \(\pm (0.52, \; 0.46)\) & once per iter. & 2.904 \(\pm (0.50, \; 0.06)\) \\
Solve                 & once per iter. & 26.83 \(\pm (2.21, \; 1.77)\) & once per iter. &  9.176 \(\pm (6.60, \; 0.61)\)\\
\hline
\end{tabular}
}% end resizebox
\caption{\small Execution time and frequency of the main operations in the MS-DD-FOM and MS-DD-ROM algorithms. Here, ``param.'' denotes the parametric configuration, ``iter.'' denotes one MS-DD-ROM iteration, ``Prec. set-up'' denotes preconditioner set-up, and ``Coarse Res.'' denotes coarse-residual evaluation. Times are reported in milliseconds.} \label{tab:execution_frequency}
\end{table}

Among the iteration-dependent operations, neural inference reduces the local solution time from approximately \(26.8\,\mathrm{ms}\) for the AMG-preconditioned conjugate-gradient solve to \(9.2\,\mathrm{ms}\), corresponding to a speedup of about \(2.9\). The surrogate transmission and coarse-residual evaluations are also faster than their exact counterparts, although the reductions are more modest because these operations are already inexpensive at the high-fidelity level.

Summing the recurring operations yields an average cost of approximately
\(35.5\,\mathrm{ms}\) per subdomain and iteration for the MS-DD-FOM and
\(15.3\,\mathrm{ms}\) for the MS-DD-ROM, corresponding to a local online
speedup of about \(2.3\). This comparison excludes the additional
parameter-dependent assembly and preconditioner costs avoided by the
surrogate method. These savings are obtained at the price of the surrogate error quantified previously.

The present comparison is intentionally performed at the subdomain level and does not constitute an end-to-end parallel benchmark. The MS-DD-FOM and MS-DD-ROM also favor different execution models: conventional local assembly and iterative solves are naturally distributed across CPU processes, whereas the neural maps can evaluate batches of subdomains concurrently on a GPU \cite{paszke2019pytorch,raina2009large}. The reported timings do not yet exploit this batched inference strategy. Consequently, a fair architecture-level comparison will require parallel implementations of both algorithms and is left for future work. In addition, the timings reported here concern the online stage and exclude data generation and network training. The value of the surrogate therefore depends on amortizing the offline cost over a sufficiently large number of parameter queries. A complete break-even analysis would additionally require accounting for high-fidelity dataset generation and for the parallel implementation of both solvers.

\section{Conclusion and Perspectives}

In this work, we introduced a computational framework that combines domain decomposition (DD) techniques with non-intrusive non-linear model order reduction for the development of scalable and computationally efficient solvers for multiscale and mixed-dimensional problems. 
We exploited localization and the definition of a reference local subproblem, defined by a \emph{local representability} assumption, to confine at the local level the computationally intensive offline part of the solver and to accelerate the execution of the iterative method used to reconstruct the global solution. In addition, the orchestration of the three surrogate models introduced to achieve full non-intrusiveness, is implemented through a \emph{cascade-training} procedure designed to generate reduced models compatible with iterative DD algorithms. 
Our analysis showed that reduced-order surrogates can be embedded within iterative DD schemes without altering the underlying convergence and stability mechanism, provided that appropriate uniform accuracy and stability properties are maintained throughout the iterative process. The numerical results for a 3D--1D blood perfusion problem illustrated the effectiveness of the proposed methodology in terms of accuracy, computational efficiency, and weak scalability.

Several questions remain open and motivate further investigation. On the theoretical side, a detailed characterization of the error introduced by the surrogate and propagated through the DD iteration is still required. This analysis should clarify the impact of the reduced approximation on transmission conditions, coarse-space corrections, and local solutions. It should also support the design of corrective mechanisms based, for instance, on residual-based error indicators or online adaptive enrichment strategies.
A further development in progress is the extension of this work beyond tensor product domains and uniform subdomain decompositions. Addressing geometrically heterogeneous subdomains would require surrogate models capable of explicit geometric generalization. This direction will follow recent advances in geometry-aware reduced-order modeling and learning-based approximation methods.

%In summary, the present study develops a principled computational approach to combining domain decomposition and reduced-order modeling in a realistic application setting. The results illustrate the feasibility of hybrid learning-augmented DD solvers and highlight key analytical and algorithmic challenges toward improving their robustness, generality, and scalability.

%Additional improvements may be obtained by incorporating structure-preserving and constraint-aware reduced-order architectures. These models could enhance robustness by enforcing conservation laws, compatibility conditions, and interfacial flux balance at the reduced level.

%From an algorithmic standpoint, the DD Robin-Robin formulation provides a natural starting point for extensions to more advanced DD strategies, including FETI and BDDC methods. Such developments would facilitate the application of the proposed framework to large-scale parallel architectures and multilevel computational settings.

\section*{Acknowledgments}
 PV and PZ acknowledge support from the MUR PRIN 2022 project 2022WKWZA8 \textit{Immersed methods for multiscale and multiphysics problems} (IMMEDIATE), part of the Next Generation EU program Mission 4, Component 2, CUP D53D23006010006; the European Union’s Euratom research and training programme 2021–2027 under grant agreement No. 101166699 \textit{Radiation safety through biological extended models and digital twins} (TETRIS);
 PZ acknowledges support from the \textit{Mathematical Challenges in Brain Mechanics} project of the Center for Advanced Study (CAS), The Norwegian Academy of Science and Letters, Oslo, Norway; ND and PZ acknowledges support from the EP PerMed JTC2024 project \textit{Targeting hypoxia with heavy ions to gain control of radioresistant cancers} (Hi-ROC), funded for Politecnico di Milano by the Ministero dell'Università e della Ricerca (MUR).
All authors acknowledge support from \textit{Dipartimento di Eccellenza} 2023--2027, Department of Mathematics, Politecnico di Milano. All authors are members of the Gruppo Nazionale per il Calcolo Scientifico (GNCS) of the Istituto Nazionale di Alta Matematica (INdAM).

%% If you have bib database file and want bibtex to generate the
%% bibitems, please use
%%
\bibliographystyle{elsarticle-num-names} 
\bibliography{mybiblio,mybiblio0,mybiblio1,mybiblio2,mybiblio3}

\section{Appendix}

\setcounter{prop}{0}
\begin{prop}[Stability under surrogate perturbations]
\label{prop:appendix_perturbed_contraction}
Let $(\mathcal V,\|\cdot\|_{\star})$ be a normed space and let
$\mathcal{DD}:\mathcal V\to\mathcal V$ be a contraction with constant
$q\in[0,1)$, that is,
\[
\left\|
 \mathcal{DD}(v)- \mathcal{DD}(w)
\right\|_{\star}
\leq
q\|v-w\|_{\star}
\qquad
\forall v,w\in\mathcal V .
\]
Assume that $u \in\mathcal V$ is the unique fixed point of
$\mathcal{DD}$, namely
\[
u =\mathcal{DD}(u).
\]
Let $(\widetilde u^{(k)})_{k\geq0}$ be the surrogate-perturbed DD iteration
defined by
\[
\widetilde u^{(k+1)}
=
 \mathcal{DD}(\widetilde u^{(k)})
+
\delta^{(k)},
\qquad k\geq0,
\]
where the perturbations satisfy
\[
\|\delta^{(k)}\|_{\star}
\leq
C_{\mathrm{DD}} \, \varepsilon
\qquad
\forall k\geq0 .
\]
Then, for every $k\geq0$,
\[
\left\|
u-\widetilde u^{(k)}
\right\|_{\star}
\leq
q^k
\left\|
u-\widetilde u^{(0)}
\right\|_{\star}
+
\frac{1-q^k}{1-q}C_{\mathrm{DD}} \, \varepsilon.
\]
In particular,
\[
\limsup_{k\to\infty}
\left\|
u-\widetilde u^{(k)}
\right\|_{\star}
\leq
\frac{C_{\mathrm{DD}}}{1-q} \, \varepsilon .
\]
\end{prop}

\bigskip

\begin{proof}
Define
\[
e^{(k)}
:=
u-\widetilde u^{(k)}.
\]
Since $u$ is the fixed point of $ \mathcal{DD}$, we have
\[
u= \mathcal{DD}(u).
\]
Using the perturbed iteration, we obtain
\[
\begin{aligned}
e^{(k+1)}
&=
u-\widetilde u^{(k+1)}
\\
&=
 \mathcal{DD}(u)
-
 \mathcal{DD}(\widetilde u^{(k)})
-
\delta^{(k)} .
\end{aligned}
\]
Hence, by the triangle inequality and the contraction property,
\[
\begin{aligned}
\|e^{(k+1)}\|_{\star}
&\leq
\left\|
 \mathcal{DD}(u)
-
 \mathcal{DD}(\widetilde u^{(k)})
\right\|_{\star}
+
\|\delta^{(k)}\|_{\star}
\\
&\leq
q\|u-\widetilde u^{(k)}\|_{\star}
+
C_{\mathrm{DD}} \, \varepsilon
\\
&=
q\|e^{(k)}\|_{\star}
+
C_{\mathrm{DD}} \, \varepsilon .
\end{aligned}
\]
Iterating this recurrence gives
\[
\|e^{(k)}\|_{\star}
\leq
q^k\|e^{(0)}\|_{\star}
+
C_{\mathrm{DD}} \, \varepsilon
\sum_{\ell=0}^{k-1}q^\ell .
\]
Since
\[
\sum_{\ell=0}^{k-1}q^\ell
=
\frac{1-q^k}{1-q},
\]
we obtain
\[
\|e^{(k)}\|_{\star}
\leq
q^k\|e^{(0)}\|_{\star}
+
\frac{1-q^k}{1-q}C_{\mathrm{DD}} \, \varepsilon.
\]
Taking the limit superior as $k\to\infty$ and using $q^k\to0$ yields
\[
\limsup_{k\to\infty}
\|e^{(k)}\|_{\star}
\leq
\frac{C_{\mathrm{DD}}}{1-q} \, \varepsilon .
\]
This concludes the proof.
\end{proof}

\bigskip

\vspace{50pt} \begin{prop}[Contraction of the discrete Robin--Robin iteration]
\label{prop:exactRR-OSMconvergence}
Let $u_h$ be the global finite element solution of the discrete problem, and let
$u_{h,l}^{(k)}$ be the exact Robin--Robin iterate on $\Omega_l$, with
$l\in\{i,j\}$, at iteration $k\geq 1$.

Assume that, for each $l\in\{i,j\}$, the bilinear form
$a_l(\cdot,\cdot)$ is symmetric, continuous, and coercive on
$\mathcal V_{h,l}$. Moreover $a_l(\cdot,\cdot)$ induces the local energy norm
\[
\|v_l\|_{a,\Omega_l}^2 := a_l(v_l,v_l),
\qquad v_l\in \mathcal V_{h,l},
\]
which is uniformly equivalent to the $H^1(\Omega_l)$-norm on
$\mathcal V_{h,l}$.

Assume moreover that the discrete lifting operators
\(
\mathcal L_l := \mathcal E_l\circ \operatorname{Tr}_{\Gamma_{ij}},
\; l\in\{i,j\},
\)
are uniformly stable with respect to the interface trace. More precisely,
there exists constants \(\ell_i\,, \, \ell_j>0\), independent of \(h\), such that
for every adjacent pair \((\Omega_i,\Omega_j)\),
\[
\|\mathcal L_j v_i\|_{a,\Omega_j}^2
\leq
\ell_j h^{-1}
\|v_i\|_{L^2(\Gamma_{ij})}^2,
\qquad
\|\mathcal L_i v_j\|_{a,\Omega_i}^2
\leq
\ell_i h^{-1}
\|v_j\|_{L^2(\Gamma_{ij})}^2,
\]
for all \(v_i\in\mathcal V_{h,i}\) and \(v_j\in\mathcal V_{h,j}\).

Define the local errors
\[
e_{h,l}^{(k)}
:=
u_{h,l}^{(k)} - u_h|_{\Omega_l},
\qquad l\in\{i,j\}.
\]
and let the Robin parameter be defined by \(\rho := \gamma h^{-1}\). If \(\gamma>0\) satisfies
\(
\gamma \geq \ell_j ,
\)
then the Robin--Robin iteration is a contraction in the norm
\[
\|e_h\|_{\star}
:=
\left(
\sum_{l\in\{i,j\}}
\|e_{h,l}\|_{a,\Omega_l}^2
+
\sum_{l\in\{i,j\}}
\frac{\gamma}{2h}
\|e_{h,l}\|_{L^2(\Gamma_{ij})}^2
\right)^{1/2}.
\]
That is, there exists $q<1$ such that
\[
\|e_h^{(k)}\|_{\star}
\leq
q\,\|e_h^{(k-1)}\|_{\star}
\qquad \forall k\geq 1.
\]
Consequently,
\[
\lim_{k\to\infty} e_{h,l}^{(k)}=0,
\qquad l\in\{i,j\},
\]
and the Robin--Robin iterates converge to the global finite element solution.
\end{prop}

\bigskip

\textit{Proof}.  We start noting that the exact solution of the global problem, $u_h$, is such that 
\begin{equation}
\begin{split}
    &a_i(u_{h}|_{\Omega_i}, v_{h,i}) +\gamma h^{-1}(u_{h}|_{\Omega_i}, v_{h,i})_{\Gamma_{ij}} =\\ &b_i(v_{h,i}) +  b_j(\mathcal{L}_j\,v_{h,i})-a_j(u_{h}|_{\Omega_j}, \mathcal{L}_j\,v_{h,i}) + \gamma h^{-1}(u_{h}|_{\Omega_j}, \mathcal{L}_j\,v_{h,i})_{\Gamma_{ij}} 
\end{split}
\end{equation}
and that a completely analogous relation holds for $\Omega_j$. Here, the Robin parameter
has been written as \(\rho := \gamma/h\).
Now, subtract the previous identity from the variational equation on the subdomain $\Omega_i$ induced by the RR iteration, namely
\begin{equation}
\begin{split}
    &a_i(u_{h,i}^{(k)}, v_{h,i}) +\gamma h^{-1}(u_{h,i}^{(k)}, v_{h,i})_{\Gamma_{ij}} =\\ &b_i(v_{h,i}) +  b_j(\mathcal{L}_j\,v_{h,i})-a_j(u_{h,j}^{(k-1)}, \mathcal{L}_j\,v_{h,i}) + \gamma h^{-1}(u_{h,j}^{(k-1)}, \mathcal{L}_j\,v_{h,i})_{\Gamma_{ij}} 
\end{split}
\end{equation}
yielding
\begin{equation}
\begin{split}
    &a_i(e_{h,i}^{(k)} , v_{h,i}) +\gamma h^{-1}(e_{h,i}^{(k)}, v_{h,i})_{\Gamma_{ij}}  =-a_j(e_{h,j}^{(k-1)}, \mathcal{L}_j\,v_{h,i}) + \gamma h^{-1}(e_{h,j}^{(k-1)}, \mathcal{L}_j\,v_{h,i})_{\Gamma_{ij}} 
\end{split}
\end{equation}

We are interested in analyzing the evolution of the error between the iterative and exact solutions, $e_{h,i}^{(k)}={u}_{h,i}^{(k)}-u_h|_{\Omega_i}$, as the iteration index $k$ increases.

By considering $v_{h,i}=e_{h,i}^{(k)}$ as particular test function, we write
\begin{align}
\label{eq:DDproof_starting-point}
a_i\bigl(e_{h,i}^{(k)},e_{h,i}^{(k)}\bigr)
\;&
\;+\;\gamma\,h^{-1}\bigl(e_{h,i}^{(k)}-e_{h,j}^{(k-1)},e_{h,i}^{(k)}\bigr)_{\Gamma_{ij}}
\;
+\,a_j\bigl(e_{h,j}^{(k-1)},\mathcal L_j\,e_{h,i}^{(k)}\bigr)=0\;;
\end{align}
an analogous relation holds, in a specular manner, for the subdomain $\Omega_j$.

At this stage, we recall the following identity, valid in any real inner product space:
\[
\langle v - w,\, v \rangle
= \frac{1}{2} \|v\|^{2} - \frac{1}{2} \|w\|^{2} + \frac{1}{2} \|v - w\|^{2}, 
\qquad \forall\, v, w \in \mathcal{V}_h.
\]

Subsequently, the second last term of ~\eqref{eq:DDproof_starting-point} can be expressed by
\[
\;\gamma\,h^{-1}\bigl(e_{h,i}^{(k)}-e_{h,j}^{(k-1)},e_{h,i}^{(k)}\bigr)_{\Gamma_{ij}}=\frac12\frac{\gamma} {h}\Bigl(\|e_{h,i}^{(k)}\|^2_{L^2(\Gamma_{ij})}-\|e_{h,j}^{(k-1)}\|^2 _{L^2(\Gamma_{ij})}+\|e_{h,i}^{(k)}-e_{h,j}^{(k-1)}\|^2_{L^2(\Gamma_{ij})} \Bigr).
\]
For the last term of~\eqref{eq:DDproof_starting-point}, we write 
\[
\begin{aligned}
a_j\bigl(e_{h,j}^{(k-1)},\,\mathcal{L}_j e_{h,i}^{(k)}\bigr)
&= a_j\bigl(e_{h,j}^{(k-1)}-\mathcal{L}_j e_{h,i}^{(k)},\,\mathcal{L}_j e_{h,i}^{(k)}\bigr)
   + a_j\bigl(\mathcal{L}_j e_{h,i}^{(k)},\,\mathcal{L}_j e_{h,i}^{(k)}\bigr) \\[10pt]
&= a_j\bigl(\mathcal{L}_j e_{h,j}^{(k-1)}-\mathcal{L}_j e_{h,i}^{(k)},\,\mathcal{L}_j e_{h,i}^{(k)}\bigr)
   + a_j\bigl(\mathcal{L}_je_{h,i}^{(k)},\,\mathcal{L}_j e_{h,i}^{(k)}\bigr)+a_j((I-\mathcal{L}_j)e_{h,j}^{(k-1)}, \mathcal{L}_je_{h,i}^{(k)}) \\[10pt]
&= \tfrac12\,a_j\bigl(\mathcal{L}_j e_{h,j}^{(k-1)},\,\mathcal{L}_j e_{h,j}^{(k-1)}\bigr)
   - \tfrac12\,a_j\bigl(\mathcal{L}_j e_{h,i}^{(k)},\,\mathcal{L}_j e_{h,i}^{(k)}\bigr) 
   \\[4pt]
&\qquad\quad+ \tfrac12\,a_j\bigl(\mathcal{L}_j e_{h,j}^{(k-1)}-\mathcal{L}_j e_{h,i}^{(k)},\,
                       \mathcal{L}_j e_{h,j}^{(k-1)}-\mathcal{L}_j e_{h,i}^{(k)}\bigr)\\[4pt]
&\qquad\quad
   + a_j\bigl(\mathcal{L}_j e_{h,i}^{(k)},\,\mathcal{L}_j e_{h,i}^{(k)}\bigr) +a_j((I-\mathcal{L}_j)e_{h,j}^{(k-1)}, \mathcal{L}_je_{h,i}^{(k)}) \\[10pt]
&= \tfrac12\,a_j\bigl(\mathcal{L}_j e_{h,j}^{(k-1)},\,\mathcal{L}_j e_{h,j}^{(k-1)}\bigr)
   + \tfrac12\,a_j\bigl(\mathcal{L}_j e_{h,i}^{(k)},\,\mathcal{L}_j e_{h,i}^{(k)}\bigr) \\[4pt]
&\qquad\quad
   + \tfrac12\,a_j\bigl(\mathcal{L}_j e_{h,j}^{(k-1)}-\mathcal{L}_j e_{h,i}^{(k)},\,
                       \mathcal{L}_j e_{h,j}^{(k-1)}-\mathcal{L}_j e_{h,i}^{(k)}\bigr)+a_j((I-\mathcal{L}_j)e_{h,j}^{(k-1)}, \mathcal{L}_je_{h,i}^{(k)}) 
\end{aligned}
\]

At this point, equality \eqref{eq:DDproof_starting-point} reduces to  

\begin{align*}
a_i\bigl(e_{h,i}^{(k)},e_{h,i}^{(k)}\bigr)&+    \tfrac{1}{2}\tfrac{\gamma} {h}\Bigl(\|e_{h,i}^{(k)}\|^2_{L^2(\Gamma_{ij})}-\|e_{h,j}^{(k-1)}\|^2 _{L^2(\Gamma_{ij})}+\|e_{h,i}^{(k)}-e_{h,j}^{(k-1)}\|^2_{L^2(\Gamma_{ij})} \Bigr)  \\[7pt]&\tfrac12\,a_j\bigl(\mathcal{L}_j e_{h,j}^{(k-1)},\,\mathcal{L}_j e_{h,j}^{(k-1)}\bigr)
   +\tfrac12 a_j({\mathcal{L}}_je_{h,i}^{(k)},{\mathcal{L}}_je_{h,i}^{(k)}) \\[4pt] &\hspace{0pt}+ \tfrac12\,a_j\bigl(\mathcal{L}_j e_{h,j}^{(k-1)}-\mathcal{L}_j e_{h,i}^{(k)},\,
                       \mathcal{L}_j e_{h,j}^{(k-1)}-\mathcal{L}_j e_{h,i}^{(k)}\bigr) +a_j((I-\mathcal{L}_j)e_{h,j}^{(k-1)}, \mathcal{L}_je_{h,i}^{(k)})=0.
\end{align*}

By the fact that the bilinear form $a_j(\cdot, \cdot)$ is SPD (it defines the inner product $\|\cdot\|_{a, \Omega_j}$) and the properties of the operator $\mathcal{L}_j$, we can write the bounds
\begin{align*}
\Bigl|\tfrac12\,a_j\bigl(\mathcal{L}_j e_{h,j}^{(k-1)}-\mathcal{L}_j e_{h,i}^{(k)},\,
                       \mathcal{L}_j e_{h,j}^{(k-1)}-\mathcal{L}_j e_{h,i}^{(k)}\bigr)\Bigr| &=  \tfrac{1}{2}\|\mathcal{L}_j e_{h,j}^{(k-1)}-\mathcal{L}_j e_{h,i}^{(k)}\|^2_{a,\Omega_j} \\[10pt] &\leq \tfrac{1}{2h}\ell_{j}\|e_{h,j}^{(k-1)}- e_{h,i}^{(k)}\|^2_{L^2(\Gamma_{ij})}   
\end{align*}
and
\begin{align*}
    |a_j((I-\mathcal{L}_j)e_{h,j}^{(k-1)}, \mathcal{L}_je_{h,i}^{(k)})| &\leq \|(I-\mathcal{L}_j)e_{h,j}^{(k-1)}\|_{a,\Omega_j}\,\|\mathcal{L}_je_{h,i}^{(k)}\|_{a,\Omega_j} \\ &\leq \|e_{h,j}^{(k-1)}\|_{a,\Omega_j}\|\mathcal{L}_j e_{h,i}^{(k)}\|_{a,\Omega_j}\leq \tfrac12\|e_{h,j}^{(k-1)}\|^2_{a,\Omega_j} + \tfrac12 \|\mathcal{L}_j e_{h,i}^{(k)}\|^2_{a,\Omega_j}
\end{align*}
then
\begin{align*}
a_i\bigl(e_{h,i}^{(k)},e_{h,i}^{(k)}\bigr)&+    \tfrac12\tfrac{\gamma} {h}\Bigl(\|e_{h,i}^{(k)}\|^2_{L^2(\Gamma_{ij})}-\|e_{h,j}^{(k-1)}\|^2 _{L^2(\Gamma_{ij})}+\|e_{h,i}^{(k)}-e_{h,j}^{(k-1)}\|^2_{L^2(\Gamma_{ij})} \Bigr)  \\[7pt]&+ \tfrac12\,a_j\bigl(\mathcal{L}_j e_{h,j}^{(k-1)},\,\mathcal{L}_j e_{h,j}^{(k-1)}\bigr)
   + \tfrac12\,a_j\bigl({\mathcal{L}}_j e_{h,i}^{(k)},\,{\mathcal{L}}_j e_{h,i}^{(k)}\bigr) 
    \\[7pt] & \leq \tfrac12\tfrac{1}{h}\ell_{j}\|e_{h,j}^{(k-1)}- e_{h,i}^{(k)}\|^2_{L^2(\Gamma_{ij})}+\tfrac12\|e_{h,j}^{(k-1)}\|^2_{a,\Omega_j} + \tfrac12 \|\mathcal{L}_j e_{h,i}^{(k)}\|^2_{a,\Omega_j}.
\end{align*}

Now, if the parameter $\gamma > 0$  is chosen such that 
\(
\gamma \geq \ell_{j} 
\),
then, expressing $a_j(v,v)$ as $\|v\|^2_{a, \Omega_j}$, we obtain
\begin{align*}
\|e_{h,i}^{(k)}\|^2_{a,\Omega_i}+\tfrac12\,\|\mathcal{L}_j e_{h,j}^{(k-1)}\|^2_{a,\Omega_j} 
   &+ \tfrac12\,\|{\mathcal{L}}_j e_{h,i}^{(k)}\|^2_{a,\Omega_j} 
\;+    \tfrac12\tfrac{\gamma} {h}\|e_{h,i}^{(k)}\|^2_{L^2(\Gamma_{ij})}  \\& \leq \tfrac12\tfrac{\gamma} {h}\|e_{h,j}^{(k-1)}\|^2 _{L^2(\Gamma_{ij})}+\tfrac12\|e_{h,j}^{(k-1)}\|^2_{a,\Omega_j} + \tfrac12 \|\mathcal{L}_j e_{h,i}^{(k)}\|^2_{a,\Omega_j}.
\end{align*}
By completely analogous steps, we obtain
\begin{align*}
\|e_{h,j}^{(k)}\|^2_{a,\Omega_j}
+ \tfrac12\,\|\mathcal{L}_i e_{h,i}^{(k-1)}\|^2_{a,\Omega_i}
&+ \tfrac12\,\|\mathcal{L}_i e_{h,j}^{(k)}\|^2_{a,\Omega_i}
+ \tfrac12\,\tfrac{\gamma}{h}\|e_{h,j}^{(k)}\|^2_{L^2(\Gamma_{ij})}
\\
&\le\;
\tfrac12\,\tfrac{\gamma}{h}\|e_{h,i}^{(k-1)}\|^2_{L^2(\Gamma_{ij})}
+ \tfrac12\,\|e_{h,i}^{(k-1)}\|^2_{a,\Omega_i}
+ \tfrac12\,\|\mathcal{L}_i e_{h,j}^{(k)}\|^2_{a,\Omega_i}.
\end{align*}
At this point, estimates above-expressed respectively on $\Omega_i$ and $\Omega_j$—are added together and like terms can be grouped, leading to 
\begin{equation}
\label{eq:start_for_convergence}
\sum_{l \in \{i,j\} }\|e_{h,l}^{(k)}\|^2_{a,\Omega_l}+\sum_{l \in \{i,j\} } \tfrac12\,\tfrac{\gamma}{h}\|e_{h,l}^{(k)}\|^2_{L^2(\Gamma_{ij})} \leq \sum_{l \in \{i,j\} } \tfrac12\,\|e_{h,l}^{(k-1)}\|^2_{a,\Omega_l}+\tfrac12\,\tfrac{\gamma}{h}\|e_{h,l}^{(k-1)}\|^2_{L^2(\Gamma_{ij})}.
\end{equation}
Having defined the norm  
\[
\|v\|_{\star,} := \left( \sum_{l \in \{i,j\}}\|v_l\|^2_{a,\Omega_l} + \sum_{l \in \{i,j\}}\tfrac{1}{2}\,\tfrac{\gamma}{h}\,\|v_l\|^2_{L^2(\Gamma_{ij})} \right)^{1/2},
\]
it is noteworthy that the inequality above can be expresses as a contraction, namely:
\begin{equation}
\label{eq:short_contraction}
\|e_h^{(k)}\|_{\star} \leq q\,\|e_h^{(k-1)}\|_{\star}, \qquad q \in  \left[\tfrac{1}{\sqrt{2}},1 \right)
\end{equation}
Indeed, inequality above can be expressed as
\[
\|e_h^{(k)}\|_* \le \sqrt{\tfrac12\sum_{l \in \{i,j\}}\|e_{h,l}^{(k-1)}\|^2_{a,\Omega_l} + \sum_{l \in \{i,j\}}\tfrac{\gamma}{2h}\|e_{h,l}^{(k-1)}\|^2_{L^2(\Gamma_{ij})}}= \sqrt{\tfrac{1}{2}A+B},
\]
having set
\[
A := \sum_{l \in \{i,j\}}\|e_{h,l}^{(k-1)}\|^2_{a,\Omega_l}, \qquad B := \sum_{l \in \{i,j\}}\tfrac{\gamma}{2h}\|e_{h,l}^{(k-1)}\|^2_{L^2(\Gamma_{ij})}.
\]
Given the identity
\[
\|e_{h}^{k-1}\|_*^2 = A + B,
\]
and introducing the parameter
\[
\theta_{k-1} :=  \frac{A}{A+B} \in (0,1],
\]
then:
\[
\tfrac12 A + B
= \left(\tfrac12\frac{A}{A+B} + \frac{B}{A+B}\right)(A+B)
= \left(1 - \frac{\theta_{k-1}}{2}\right)\|e^{k-1}_h\|_*^2.
\]
Substituting the equality above in into~\eqref{eq:short_contraction} gives
\[
\|e^{(k)}_h\|_*
\;\le\;
\sqrt{1 - \frac{\theta_{k-1}}{2}}\;\|e^{(k-1)}_h\|_*.
\tag{3}
\]
Since \(0 < \theta_{k-1} \leq 1\), we have
\[
\frac{1}{\sqrt{2}} \;\le\; q \;<\; 1.
\]
Having $q < 1$,implies that the error vanishes in the limit \(k \to \infty\), thus concluding the proof. \hfill $\square$

\end{document}